\documentclass[11pt]{amsart}

\usepackage[a4paper,margin=1in]{geometry}
\usepackage{amsmath,amssymb,amsthm,mathtools}
\numberwithin{equation}{section}
\usepackage{bm}
\usepackage{microtype}
\usepackage{xcolor}
\usepackage[colorlinks=true,linkcolor=blue!45!black,citecolor=blue!45!black,urlcolor=blue!55!black]{hyperref}

\usepackage{aliascnt}
\usepackage[nameinlink,capitalize,noabbrev]{cleveref}

\hypersetup{
    pdftitle={Global Strict Monotonicity and Asymptotics of the First Steklov Eigenvalue for Regular Polygons},
    pdfauthor={Zhuo Cheng, Changfeng Gui, Yeyao Hu, Qinfeng Li, and Ruofei Yao}
}

\theoremstyle{plain}
\newtheorem{theorem}{Theorem}[section]

\newaliascnt{proposition}{theorem}
\newtheorem{proposition}[proposition]{Proposition}
\aliascntresetthe{proposition}

\newaliascnt{lemma}{theorem}
\newtheorem{lemma}[lemma]{Lemma}
\aliascntresetthe{lemma}

\newaliascnt{corollary}{theorem}
\newtheorem{corollary}[corollary]{Corollary}
\aliascntresetthe{corollary}

\newtheorem{conjecture}{Conjecture}[section]

\theoremstyle{definition}
\newaliascnt{definition}{theorem}
\newtheorem{definition}[definition]{Definition}
\aliascntresetthe{definition}

\theoremstyle{remark}
\newaliascnt{remark}{theorem}
\newtheorem{remark}[remark]{Remark}
\aliascntresetthe{remark}

\crefname{theorem}{Theorem}{Theorems}
\Crefname{theorem}{Theorem}{Theorems}
\crefname{proposition}{Proposition}{Propositions}
\Crefname{proposition}{Proposition}{Propositions}
\crefname{lemma}{Lemma}{Lemmas}
\Crefname{lemma}{Lemma}{Lemmas}
\crefname{corollary}{Corollary}{Corollaries}
\Crefname{corollary}{Corollary}{Corollaries}
\crefname{conjecture}{Conjecture}{Conjectures}
\Crefname{conjecture}{Conjecture}{Conjectures}
\crefname{remark}{Remark}{Remarks}
\Crefname{remark}{Remark}{Remarks}

\newcommand{\D}{\mathbb D}
\newcommand{\T}{\mathbb T}
\newcommand{\Z}{\mathbb Z}
\newcommand{\R}{\mathbb R}
\newcommand{\C}{\mathbb C}
\newcommand{\Per}{\operatorname{Per}}
\newcommand{\Span}{\operatorname{span}}
\newcommand{\HS}{\mathrm{HS}}
\newcommand{\eps}{\varepsilon}

\newcommand{\OmN}{\Omega_N}
\newcommand{\sig}{\sigma}
\newcommand{\sigone}{\sigma_1}
\newcommand{\Lam}{\Lambda_{\D}}
\newcommand{\EN}{\mathcal E}
\newcommand{\BN}{\mathcal B_N}
\newcommand{\HN}[2]{\mathcal H_{#1,#2}}

\title[Global Strict Monotonicity for Regular Polygons]{\bfseries Global Strict Monotonicity and Asymptotics of the First Steklov Eigenvalue for Regular Polygons}

\author{Zhuo Cheng}
\address{School of Mathematics and Statistics, HNP-LAMA, Central South University, Changsha, Hunan 410083, P. R. China}
\email{zhuo@outlook.cz}

\author{Changfeng Gui}
\address{Department of Mathematics, University of Macau, Macau SAR, P. R. China}
\address{Zhuhai UM Science and Technology Research Institute, Hengqin, Guangdong 519031, P. R. China}
\email{changfenggui@um.edu.mo}

\author{Yeyao Hu}
\address{School of Mathematics and Statistics, HNP-LAMA, Central South University, Changsha, Hunan 410083, P. R. China}
\email{huyeyao@gmail.com}

\author{Qinfeng Li}
\address{School of Mathematics, Hunan University, Changsha, P. R. China}
\email{liqinfeng1989@gmail.com}

\author{Ruofei Yao}
\address{School of Mathematics, South China University of Technology, Guangzhou, P. R. China}
\email{ruofeiyaopde@gmail.com}

\date{}

\begin{document}
% ============================================================

\begin{abstract}
For a bounded Lipschitz planar domain $\Omega$, let $\sigone(\Omega)$ denote its first nonzero Steklov eigenvalue.  Let $\Omega_N$ be the regular $N$-gon normalized to have perimeter $2\pi$.  We prove that
\[
        \sigone(\Omega_{N+1})>\sigone(\Omega_N),\qquad N\ge3.
\]
Consequently, $\sigone(\Omega_N)$ increases strictly to the disk value $1$.  We also show that the real first eigenspace has dimension two and derive the asymptotic expansion
\[
        \sigone(\Omega_N)
        =1-\frac{2\zeta(3)}{N^3}
          -\frac{8\zeta(4)}{N^4}
          -\frac{26\zeta(5)}{N^5}
          +O(N^{-6}).
\]
The proof combines an equivariant Schwarz--Christoffel pullback, a nodal selection of the two critical Fourier residue classes, and a positivity-improving Perron comparison for the associated reciprocal operators.  The asymptotic expansion follows independently from a Schur reduction and the evaluation of Euler-type sums.
\end{abstract}

\maketitle

% ============================================================

\section{Introduction and main results}
\label{sec:introduction}
% ============================================================

Let $\Omega\subset\R^2$ be a bounded Lipschitz domain.  The Steklov problem, introduced by Stekloff in connection with boundary heat transfer, is
\begin{equation}\label{eq:steklov}
\begin{cases}
        \Delta u=0, & x\in\Omega,\\
        \partial_\nu u=\sigma u, & x\in\partial\Omega.
\end{cases}
\end{equation}
Constant functions give the eigenvalue $0$.  The remaining eigenvalues form a discrete sequence
\[
        0=\sig_0(\Omega)<\sigone(\Omega)\le \sig_2(\Omega)\le\cdots,
\]
counted with multiplicity.  Throughout the paper, $\sigone(\Omega)$ denotes the first nonzero Steklov eigenvalue of $\Omega$.  Equivalently, the boundary trace is an eigenfunction of the Dirichlet-to-Neumann operator on $\partial\Omega$, and the first positive eigenvalue is characterized by
\begin{equation}\label{eq:domain-rayleigh}
\sigone(\Omega)=
        \inf_{\substack{u\in H^1(\Omega)\setminus\{0\}\\
        \int_{\partial\Omega}u\,ds=0}}
        \frac{\int_\Omega |\nabla u|^2\,dx}
             {\int_{\partial\Omega}|u|^2\,ds}.
\end{equation}
For the spectral-geometric background, the Dirichlet-to-Neumann formulation, and standard open problems, we refer to \cite{ColboisGirouardGordonSher2024,GirouardPolterovich2016,GirouardPolterovich2017}; the original source is \cite{Stekloff1902}.

In two dimensions, the quotient in \eqref{eq:domain-rayleigh} separates two effects: the Dirichlet energy is conformally invariant, whereas the denominator records the boundary metric density.  This observation underlies the classical conformal approach to planar Steklov problems \cite{HerschPayneSchiffer1974,GirouardLaugesenSiudeja2016} and identifies
\[
        \sigone(\Omega)\Per(\Omega)
\]
as the natural scale-invariant quantity.  We normalize the regular $N$-gon by
\begin{equation}\label{eq:perimeter-normalization}
        \Per(\OmN)=2\pi.
\end{equation}
Under this normalization, the limiting disk is the unit disk and has first nonzero Steklov eigenvalue $1$.

Our main result gives a strict comparison along the family of regular polygons.

\begin{theorem}\label{thm:global-monotonicity}
Let $\Omega_N$ denote the regular $N$-gon normalized by $\Per(\Omega_N)=2\pi$.  Then, for every $N\ge3$,
\begin{equation}\label{eq:main-monotonicity}
        \sigone(\Omega_{N+1})>\sigone(\Omega_N).
\end{equation}
\end{theorem}

The Steklov problem has been studied from several complementary viewpoints.  Precise spectral asymptotics exhibit the boundary character of the Dirichlet-to-Neumann operator and its relation to boundary geometry \cite{ColboisGirouardGordonSher2024,GirouardParnovskiPolterovichSher2014}.  In parallel, the theory of nodal sets and nodal domains describes the geometry of Steklov eigenfunctions both near and away from the boundary \cite{BellovaLin2015,SoggeWangZhu2016,WangZhu2015,Zelditch2015,Zhu2016,Zhu2020}.  The surveys and problem lists \cite{ColboisGirouardGordonSher2024,GirouardPolterovich2016,GirouardPolterovich2017} provide further background and place these developments in a broader spectral-geometric context.

Fraser and Schoen established a fundamental connection between Steklov eigenvalues and free-boundary minimal surfaces.  They showed that the coordinate functions of a free-boundary minimal immersion into the Euclidean unit ball are Steklov eigenfunctions with eigenvalue $1$, thereby linking extremal Steklov metrics to minimal-surface geometry \cite{FraserSchoen2011}.  Their subsequent work produced sharp eigenvalue bounds and rigidity results, identified extremal roles for the critical catenoid and the critical M\"obius band, and developed a broader shape-optimization theory for Steklov eigenvalues \cite{FraserSchoen2013,FraserSchoen2016,FraserSchoen2019}.  This viewpoint complements the analytic and variational approaches to the Steklov spectrum.

The isoperimetric theory begins with Weinstock's classical theorem, which states that the disk uniquely maximizes $\sigone\Per$ among simply connected planar domains \cite{Weinstock1954}.  Brock proved a sharp higher-dimensional inequality under a volume constraint, with the ball as optimizer \cite{Brock2001}.  Under fixed boundary measure, however, Fraser and Schoen showed that the ball need not maximize the first nonzero Steklov eigenvalue among contractible domains \cite{FraserSchoen2019}.  Bucur, Ferone, Nitsch, and Trombetti recovered a sharp Weinstock-type inequality in the convex class \cite{BucurFeroneNitschTrombetti2021}, while Bucur and Nahon exhibited an instability phenomenon for the planar Weinstock inequality without additional geometric control \cite{BucurNahon2021}.  For higher eigenvalues, the inequalities of Hersch--Payne--Schiffer and later optimization and homogenization results reveal a more intricate extremal picture \cite{FraserSchoen2020,GirouardLagace2021,HerschPayneSchiffer1974}.  These results motivate the study of structured subclasses for which sharper comparisons may be available.

For each $N\ge3$, let
\[
        \mathcal P_N:=
        \left\{\Omega\subset\R^2:\,
        \begin{array}{l}
        \Omega \text{ is a convex polygon with at most } N \text{ sides},\\
        \Per(\Omega)=2\pi
        \end{array}
        \right\}.
\]
The corresponding polygonal P\'olya--Szeg\H{o} principle for the Steklov problem is the following conjecture.

\begin{conjecture}[Steklov polygonal P\'olya--Szeg\H{o} conjecture]\label{conj:steklov-polya-szego}
For every $N\ge3$ and every $\Omega\in\mathcal P_N$,
\[
        \sigma_1(\Omega)\le \sigma_1(\Omega_N),
\]
with equality if and only if $\Omega$ is congruent to the regular $N$-gon $\Omega_N$.
\end{conjecture}

\Cref{conj:steklov-polya-szego} appears, in essentially this form, in the AIM Steklov problem list and workshop report \cite{AIMSteklov2018}; we record it only as motivation.  It would imply \Cref{thm:global-monotonicity} immediately, since $\Omega_N\in\mathcal P_{N+1}$ but $\Omega_N$ is not congruent to $\Omega_{N+1}$.  \Cref{thm:global-monotonicity} proves this one-parameter consequence without assuming the full polygonal extremal conjecture.

Results for the Dirichlet Laplacian provide a useful comparison and partly motivate the present work.  A recent breakthrough by D\"ahne, G\'omez-Serrano, and Pech-Alberich \cite{DahneGomezSerranoPechAlberich2026} established the strict monotonicity of the first Dirichlet eigenvalue along the family of regular polygons.  Their proof develops a computer-assisted scheme showing that the first Dirichlet eigenvalue of the regular $N$-gon decreases strictly with $N$.  This monotonicity is predicted by the classical polygonal P\'olya--Szeg\H{o} conjecture, which asserts that the regular $N$-gon uniquely minimizes the first Dirichlet eigenvalue among area-normalized polygons with at most $N$ sides.  P\'olya proved the triangular and quadrilateral cases by Steiner symmetrization \cite{Polya1948,PolyaSzego1951}, and the numerical study of Antunes and Freitas \cite{AntunesFreitas2006} provides strong evidence for the general conjecture.  Indrei \cite{Indrei2024} proved local minimality of the regular $N$-gon for sufficiently large $N$ and quantitative stability in the triangular class.  Bogosel and Bucur \cite{BogoselBucur2024} established local minimality for $N=5,6,7,8$ by a hybrid analytic--numerical argument.  High-order asymptotic expansions for regular polygons were obtained in \cite{BerghausGeorgievMonienRadchenko2024}.

In the Dirichlet setting, monotonicity can be studied through local geometric modifications.  The surgery method described in Henrot's monograph \cite{Henrot2006}, for example, inserts an additional vertex while decreasing the principal eigenvalue, and related deformation and flow arguments exploit similar local variations \cite{ChenGuiYao2026,HuangLiXieYang2025}.  Such mechanisms are not directly available here: the perimeter-normalized regular polygons $\Omega_N$ and $\Omega_{N+1}$ are nonnested, while changing $N$ globally changes the boundary metric and the nonlocal Dirichlet-to-Neumann operator.  These differences call for a comparison principle on a common reference domain.

Polygonal domains also have a distinct Steklov spectral geometry.  For smooth compact surfaces with boundary, the high-frequency spectrum is governed, up to the remainder scale in \cite{GirouardParnovskiPolterovichSher2014}, by boundary-length data; for smooth simply connected planar domains, this datum reduces to the perimeter.  For curvilinear polygons, side lengths and angles enter the refined asymptotics, and inverse results recover such data under the hypotheses of \cite{KrymskiLevitinParnovskiPolterovichSher2021,LevitinParnovskiPolterovichSher2022}.  For straight convex polygons, recent work establishes angle-dependent bounds and spectral finiteness under the stated admissibility or genericity assumptions, and a subsequent preprint studies spectral determination in several polygonal classes \cite{DrydenGordonMorenoRowlettVillegasBlas2025,DrydenGordonMorenoRowlettVillegasBlas2026}.  These results are not used below, but they explain why regular polygons form a natural one-parameter family in which the number of corners varies and the limiting domain is the disk.

The proof of \Cref{thm:global-monotonicity} does not compare the polygons as planar subsets.  The domains are not nested, their vertices depend on $N$, and smooth shape calculus does not apply to this polygonal family.  We instead place every polygon on the same boundary by an equivariant Schwarz--Christoffel pullback.  The boundary regularity of the conformal map is compatible with the classical Hardy-space boundary theory of F. and M. Riesz and its standard exposition in Duren's monograph \cite{Duren1970,Riesz1923,RieszRiesz1916}; the explicit polygonal map and its vertex singularities are standard in Schwarz--Christoffel theory \cite{DriscollTrefethen2002,Pommerenke1992}.  The Steklov equation becomes the weighted boundary problem
\[
        \Lam g=\sigma w_N g,
        \qquad
        w_N(\theta)=C_N|1-e^{iN\theta}|^{-2/N},
\]
where $\Lam$ is the Dirichlet-to-Neumann operator of the unit disk.  The singular weight records the vertices and has Fourier support in $N\Z$.  The boundary space therefore splits into residue classes modulo $N$.  We first determine which classes can carry a first trace.

Let $\T=\R/(2\pi\Z)$.  Here and below, $H^{1/2}(\T)$ is complex-valued unless real-valued functions are explicitly specified.  Write $\widehat g(k)$ for the $k$-th Fourier coefficient of $g$.  For $0\le r\le N-1$ set
\begin{equation}\label{eq:intro-block-space}
\HN{N}{r}
:=
\left\{
        g\in H^{1/2}(\T):
        \widehat g(k)=0
        \text{ whenever } k\not\equiv r\pmod N
\right\}.
\end{equation}
Equivalently, $\HN{N}{r}$ is the closed linear span in $H^{1/2}(\T)$ of the modes
$e^{i(r+mN)\theta}$, $m\in\Z$.  These subspaces are the Fourier residue blocks associated with the Schwarz--Christoffel pullback.  The next theorem records both the localization and the multiplicity structure of the first Steklov level.
\begin{theorem}\label{thm:first-level-structure}
Let $\OmN$ be the regular $N$-gon normalized by \eqref{eq:perimeter-normalization}.  After the Schwarz--Christoffel pullback to $\T$, every boundary trace of a first nonzero Steklov eigenfunction belongs to
\[
        \HN{N}{1}\oplus\HN{N}{N-1}.
\]
Moreover, after complexification, each of the two critical blocks contributes exactly one complex eigenline, and the two lines are interchanged by complex conjugation.  Consequently, if $E_{\sigone}(\OmN)$ denotes the real eigenspace associated with the first nonzero Steklov eigenvalue, then
\begin{equation}\label{eq:first-eigenspace-dimension}
        \dim_{\R} E_{\sigone}(\OmN)=2.
\end{equation}
\end{theorem}

In terms of the block Rayleigh values, the localization statement is equivalent to
\begin{equation}\label{eq:critical-block-comparison}
        \sigone(\OmN)
        =\sig_1^{(1)}(N)
        =\sig_1^{(N-1)}(N)
        <\sig_1^{(r)}(N),
        \qquad r=0\ \text{or}\ 2\le r\le N-2.
\end{equation}
Here $\sig_1^{(r)}(N)$ denotes the first Rayleigh value in the $r$-th residue block.  The complex line in $\HN{N}{1}$ is identified explicitly in Section~\ref{sec:first-level-structure} as $\C\psi_N$, and its conjugate is the line in $\HN{N}{N-1}$.

Nodal information enters precisely in the localization statement.  Symmetry yields the block decomposition but does not identify the minimizing blocks.  In Section~\ref{sec:first-level-structure} we prove that every real first trace has exactly two boundary nodal domains, one positive and one negative.  The proof combines the Steklov Courant theorem, available on bounded Lipschitz domains, with a topological Euler boundary count.  The Euler count itself only requires a finite properly embedded nodal forest in a topological disk.  Appendix~\ref{app:nodal-input} verifies this finite-forest structure for the regular polygons considered here, using Cheng's local structure theorem for interior nodal sets \cite{Cheng1976} and the Kondrat'ev--Grisvard--Dauge/Costabel--Dauge corner theory \cite{Grisvard1985,CostabelDauge1993,CostabelDauge1994,Dauge1988,Kondratev1967}, and then applies the abstract Euler lemma.  Quotient periodicity excludes the zero block, while the middle blocks are excluded by a low-mode obstruction.  Further context on Steklov nodal sets is provided by \cite{BellovaLin2015,SoggeWangZhu2016,WangZhu2015,Zelditch2015,Zhu2016,Zhu2020}.

This localization is specific to the first Steklov level.  An extension to higher eigenvalues would require a replacement for the boundary nodal theorem, since higher eigenspaces need not have exactly one positive and one negative boundary component.  This is the main obstruction to applying the present argument beyond $\sigone$.

Once the first level is localized to the two critical blocks, the argument becomes order-theoretic.  On either block, the reciprocal Steklov problem is represented on a common energy-normalized $\ell^2(\Z)$ space by a compact self-adjoint Hilbert--Schmidt operator $A_N$ with strictly positive matrix entries.  We prove the entrywise inequality
\[
        0<(A_{N+1})_{m,n}\le (A_N)_{m,n},
\]
with equality in the non-strict inequality only at $(m,n)=(0,0)$.  This positivity has two consequences.  An entrywise comparison followed by a positivity-improving Perron argument yields the strict monotonicity in \Cref{thm:global-monotonicity}; Perron simplicity shows that each critical block contributes exactly one complex eigenline, completing the structure theorem.

The monotonicity proof is independent of the asymptotic analysis.  The selected critical branch nevertheless admits an explicit expansion that quantifies convergence to the disk.  The final coefficient extraction reduces to classical Euler-sum identities; we use \cite{FlajoletSalvy1998} as a standard reference and \cite{BorweinBradleyBroadhurst1997} for related evaluations.

\begin{theorem}\label{thm:order-five-asymptotics}
As $N\to\infty$,
\begin{equation}\label{eq:order-five-asymptotics}
        \sigone(\OmN)
        =1-\frac{2\zeta(3)}{N^3}
          -\frac{8\zeta(4)}{N^4}
          -\frac{26\zeta(5)}{N^5}
          +O(N^{-6}).
\end{equation}
\end{theorem}

\begin{remark}\label{rem:role-of-asymptotics}
The Schur expansion analyzes the selected critical branch quantitatively and is independent of the proof of \Cref{thm:global-monotonicity}.  In particular, \eqref{eq:order-five-asymptotics} alone gives
\[
        \sigone(\Omega_{N+1})-\sigone(\Omega_N)
        =\frac{6\zeta(3)}{N^4}+O(N^{-5}),
\]
which is positive for all sufficiently large $N$.  Moreover, the $O(N^{-6})$ remainder in \eqref{eq:order-five-asymptotics} is bounded in absolute value by $1200/N^6$.  Thus the quantitative estimate used in the earlier arXiv version remains valid; the present version simply omits the auxiliary computation, since the proof given here is complete and entirely analytic.  Together with the global comparison and Weinstock's upper bound, the expansion gives quantitative monotone convergence to the disk.  Hence the asymptotic argument yields eventual strict monotonicity with explicit remainder control, whereas \Cref{thm:global-monotonicity} establishes strict monotonicity for every $N\ge3$.
\end{remark}

The proof has three main stages.  Sections~\ref{sec:sc-pullback}--\ref{sec:blocks} construct the weighted residue-block model.  Section~\ref{sec:first-level-structure} proves \Cref{thm:first-level-structure}: nodal topology selects the two critical blocks, and Perron simplicity determines the first eigenspace.  Section~\ref{sec:perron-monotonicity} then proves \Cref{thm:global-monotonicity} by comparing the critical reciprocal operators.  Finally, Section~\ref{sec:schur-asymptotics} develops the independent Schur expansion used for \Cref{thm:order-five-asymptotics}.

% ============================================================
\section{Schwarz--Christoffel pullback and the weighted disk model}
\label{sec:sc-pullback}
% ============================================================

We place all polygonal problems on the fixed boundary $\T$.  Because the Schwarz--Christoffel density is singular at the prevertices, we use the pullback only for Dirichlet energies, traces, and weak boundary forms.  The Fourier coefficients of $w_N$ are computed by radial regularization in Section~\ref{sec:blocks}; subsequent coefficient identities are first verified for finite Fourier sums and then extended by continuity of the weighted pairing on $H^{1/2}(\T)$.  No pointwise Fourier expansion of the singular density is used.

We identify the unit circle with $\T=\R/(2\pi\Z)$ and use the Fourier convention
\begin{equation}\label{eq:fourier-convention}
        \widehat f(k)=\frac1{2\pi}\int_0^{2\pi} f(\theta)e^{-ik\theta}\,d\theta.
\end{equation}
The Dirichlet-to-Neumann operator on the unit disk satisfies
\begin{equation}\label{eq:dtn-fourier}
        \Lam e^{ik\theta}=|k|e^{ik\theta}.
\end{equation}
Throughout, inner products and sesquilinear forms are linear in the first argument and conjugate-linear in the second.

\subsection{Equivariant Schwarz--Christoffel map}

\begin{lemma}\label{lem:sc-normalization}
For each $N\ge3$, fix the center and scale.  Then there exists a conformal map $f_N:\D\to\OmN$, unique up to a simultaneous rotation, such that
\begin{equation}\label{eq:sc-derivative}
        f_N'(z)=C_N(1-z^N)^{-2/N},\qquad C_N>0.
\end{equation}
The radial boundary density, defined for almost every $\theta$ by
\begin{equation}\label{eq:wN-definition}
        w_N(\theta):=\lim_{r\to1^-}|f_N'(re^{i\theta})|,
\end{equation}
satisfies
\begin{equation}\label{eq:wN-explicit}
        w_N(\theta)=C_N|1-e^{iN\theta}|^{-2/N},
        \qquad
        \frac1{2\pi}\int_0^{2\pi}w_N(\theta)\,d\theta=1,
\end{equation}
and is $2\pi/N$-periodic.  Moreover, $w_N\in L^p(\T)$ and $|f_N'|^2\in L^p(\D)$ for every $1\le p<N/2$, while $w_N\notin L^{N/2}(\T)$.
\end{lemma}

\begin{proof}
Center the regular polygon at the origin and define
\[
        R(z)=e^{2\pi i/N}z,\qquad S(w)=e^{2\pi i/N}w.
\]
Choose the Riemann map $F:\D\to\OmN$ normalized by
\[
        F(0)=0,\qquad F'(0)\in(0,\infty).
\]
Because $S\OmN=\OmN$, the map
\[
        G(z):=S^{-1}F(Rz)
\]
is also a Riemann map from $\D$ onto $\OmN$ and satisfies
\[
        G(0)=0,\qquad G'(0)=S^{-1}F'(0)R=F'(0)\in(0,\infty).
\]
By uniqueness of the normalized Riemann map, $G=F$.  Hence
\begin{equation}\label{eq:equivariant-F}
        F(e^{2\pi i/N}z)=e^{2\pi i/N}F(z).
\end{equation}
By Carath\'eodory's theorem, $F$ extends homeomorphically to the boundary.  Equation \eqref{eq:equivariant-F} shows that the prevertices form a single $N$-fold orbit on $\partial\D$.  Rotating the source circle and compensating by a rotation of the target, we may take
\[
        \zeta_\ell=e^{2\pi i\ell/N},\qquad \ell=0,\ldots,N-1.
\]
The Schwarz--Christoffel derivative formula for a polygonal image of the disk gives
\[
        F'(z)=C\prod_{\ell=0}^{N-1}(z-\zeta_\ell)^{-2/N},
\]
because every interior angle equals $(1-2/N)\pi$ and hence every exterior-angle exponent is $-2/N$; see, for example, \cite[Chapter~2]{DriscollTrefethen2002}.  Since
\[
        \prod_{\ell=0}^{N-1}(z-\zeta_\ell)=z^N-1=-(1-z^N),
\]
we have, for the branch fixed above,
\[
        (z^N-1)^{-2/N}=\eta_N(1-z^N)^{-2/N},
        \qquad |\eta_N|=1.
\]
Absorb the unimodular factor \(\eta_N\) into a rotation of the target.  Since this rotation is an isometry, it changes neither the perimeter normalization nor the Steklov spectrum.  With the analytic branch normalized to equal $1$ at $z=0$, we obtain
\[
        f_N'(z)=C_N(1-z^N)^{-2/N},\qquad C_N>0.
\]
The normalization $\Per(\OmN)=2\pi$ is equivalent to the radial $L^1$ identity
\[
        \lim_{r\to1^-}\int_0^{2\pi}|f_N'(re^{i\theta})|\,d\theta
        =\int_0^{2\pi}w_N(\theta)\,d\theta=2\pi.
\]
Near a prevertex, $w_N(\theta)\asymp|\theta-\theta_0|^{-2/N}$ and $|f_N'(z)|^2\asymp|z-\zeta|^{-4/N}$.  The one- and two-dimensional power tests give the range $p<N/2$, with divergence for $w_N$ at $p=N/2$.  Since $w_N\in L^1$ for $N\ge3$, the perimeter identity yields the normalization in \eqref{eq:wN-explicit}.  The dependence on $e^{iN\theta}$ gives the $2\pi/N$-periodicity.
\end{proof}

\subsection{Weighted boundary forms}

We use the following weak formulation of \eqref{eq:steklov}.  A nonzero function $u\in H^1(\Omega)$ is a Steklov eigenfunction with eigenvalue $\sigma$ if and only if
\begin{equation}\label{eq:steklov-weak}
        \int_\Omega \nabla u\cdot\overline{\nabla v}\,dx
        =\sigma\int_{\partial\Omega}
        (\operatorname{Tr}u)\,\overline{(\operatorname{Tr}v)}\,ds,
        \qquad v\in H^1(\Omega).
\end{equation}
Testing \eqref{eq:steklov-weak} against $H^1_0(\Omega)$ gives the weak harmonicity of $u$, and the full identity encodes the boundary condition.  The form \eqref{eq:steklov-weak} is well suited to conformal pullback: the Dirichlet form is invariant in two dimensions, whereas the boundary pairing acquires the density $w_N$.  We record the required energy, measure, and trace identifications.

\begin{lemma}\label{lem:homogeneous-dirichlet-isometry}
Let $f:\D\to\Omega$ be a conformal map onto a bounded Jordan polygon.  Composition with $f$ induces an isometric isomorphism between the homogeneous Dirichlet spaces modulo constants:
\[
        \dot H^1(\Omega)/\C \longleftrightarrow \dot H^1(\D)/\C,
        \qquad u\longmapsto u\circ f,
\]
in the sense that
\begin{equation}\label{eq:homogeneous-dirichlet-isometry}
        \int_\Omega |\nabla u|^2\,dx
        =\int_\D |\nabla(u\circ f)|^2\,dx
\end{equation}
whenever either side is finite.  The same statement holds for the inverse conformal map.
\end{lemma}

\begin{proof}
For $0<r<1$, the restriction $f:\D_r\to f(\D_r)$ is a smooth conformal diffeomorphism.  The Sobolev chain rule and change of variables therefore give
\[
        \int_{\D_r}|\nabla(u\circ f)|^2
        =\int_{f(\D_r)}|\nabla u|^2,
\]
because $Df=|f'|R$ for a rotation $R$ and $J_f=|f'|^2$.  As $r\to1^-$, we have $\D_r\uparrow\D$ and $f(\D_r)\uparrow\Omega$.  Monotone convergence yields \eqref{eq:homogeneous-dirichlet-isometry}.  Applying the same argument to $f^{-1}$ proves the isometric isomorphism modulo constants.
\end{proof}

Since $\OmN$ is a bounded Lipschitz Jordan domain, the boundary correspondence theorem for conformal maps onto rectifiable Jordan domains implies that
\[
        \gamma_N(\theta):=f_N(e^{i\theta})
\]
is absolutely continuous and satisfies
\[
        |\gamma_N'(\theta)|=|f_N'(e^{i\theta})|=w_N(\theta)
        \qquad\text{for a.e. }\theta;
\]
see, for instance, \cite[Chapter~2]{DriscollTrefethen2002}.  Hence the one-dimensional change-of-variables formula gives
\begin{equation}\label{eq:boundary-measure-pushforward-general}
        \int_{\partial\OmN}\Phi\,ds
        =\int_0^{2\pi}\Phi(\gamma_N(\theta))w_N(\theta)\,d\theta
\end{equation}
for every continuous $\Phi$ on $\partial\OmN$.  Thus composition with $\gamma_N$ extends to the unitary map
\begin{equation}\label{eq:L2-boundary-unitary}
        L^2(\partial\OmN,ds)
        \longrightarrow L^2(\T,w_N(\theta)d\theta),
        \qquad \varphi\longmapsto \varphi\circ\gamma_N.
\end{equation}

\begin{proposition}\label{prop:sobolev-pullback-trace}
Composition with $f_N$ and $f_N^{-1}$ defines mutually inverse bounded linear maps
\[
        H^1(\OmN)\longleftrightarrow H^1(\D),
        \qquad
        u\longmapsto u\circ f_N,
        \qquad
        V\longmapsto V\circ f_N^{-1}.
\]
These maps are compatible with Sobolev traces.  More precisely, for every $V\in H^1(\D)$,
\begin{equation}\label{eq:disk-to-polygon-trace}
        \operatorname{Tr}_{\partial\OmN}(V\circ f_N^{-1})
        =(\operatorname{Tr}_{\T}V)\circ\gamma_N^{-1}
        \qquad\text{in }L^2(\partial\OmN,ds),
\end{equation}
and for every $u\in H^1(\OmN)$,
\begin{equation}\label{eq:polygon-to-disk-trace}
        \operatorname{Tr}_{\T}(u\circ f_N)
        =(\operatorname{Tr}_{\partial\OmN}u)\circ\gamma_N
        \qquad\text{in }L^2(\T,w_N(\theta)d\theta).
\end{equation}
\end{proposition}

\begin{proof}
Let $V\in H^1(\D)$ and choose $V_j\in C^\infty(\overline\D)$ with $V_j\to V$ in $H^1(\D)$.  Set $v_j=V_j\circ f_N^{-1}$.  By \Cref{lem:homogeneous-dirichlet-isometry}, the Dirichlet seminorms converge.  Choose $1<p<N/2$ and let $p'=p/(p-1)$.  Since $|f_N'|^2\in L^p(\D)$ by \Cref{lem:sc-normalization}, H\"older's inequality and the embedding $H^1(\D)\hookrightarrow L^{2p'}(\D)$ give
\[
        \|(V_j-V)\circ f_N^{-1}\|_{L^2(\OmN)}^2
        =\int_{\D}|V_j-V|^2|f_N'|^2
        \le C\|V_j-V\|_{H^1(\D)}^2.
\]
Hence $v_j\to v:=V\circ f_N^{-1}$ in $H^1(\OmN)$.  For smooth $V_j$, the boundary identity
\[
        \operatorname{Tr}_{\partial\OmN}v_j
        =(\operatorname{Tr}_{\T}V_j)\circ\gamma_N^{-1}
\]
holds away from the finitely many vertices.  The trace theorem, \Cref{lem:sc-normalization}, the embedding $H^{1/2}(\T)\hookrightarrow L^{2p'}(\T)$, and the unitary map \eqref{eq:L2-boundary-unitary} allow passage to the limit, proving \eqref{eq:disk-to-polygon-trace}.

Conversely, let $u\in H^1(\OmN)$ and choose $u_j\in C^\infty(\mathbb R^2)|_{\overline{\OmN}}$ with $u_j\to u$ in $H^1(\OmN)$.  Set $U_j=u_j\circ f_N$.  Again \Cref{lem:homogeneous-dirichlet-isometry} controls the Dirichlet seminorms.  From \eqref{eq:sc-derivative},
\[
        |f_N'(z)|=C_N|1-z^N|^{-2/N}\ge C_N2^{-2/N},
\]
so $(f_N^{-1})'\in L^\infty(\OmN)$.  Therefore
\[
        \|U_j-u\circ f_N\|_{L^2(\D)}^2
        =\int_{\OmN}|u_j-u|^2|(f_N^{-1})'|^2
        \le C\|u_j-u\|_{L^2(\OmN)}^2,
\]
and $U_j\to U:=u\circ f_N$ in $H^1(\D)$.  For smooth $u_j$,
\[
        \operatorname{Tr}_{\T}U_j
        =(\operatorname{Tr}_{\partial\OmN}u_j)\circ\gamma_N
\]
holds away from the prevertices.  Passing to the limit as above, using also the unitary identification \eqref{eq:L2-boundary-unitary}, gives \eqref{eq:polygon-to-disk-trace}.  The two composition operators are inverse to each other almost everywhere, and the preceding estimates show that both are bounded on $H^1$.
\end{proof}

We now formulate the pulled-back problem on the fixed boundary $\T$.  For $g,h\in H^{1/2}(\T)$, define
\begin{equation}\label{eq:E-form}
        \EN(g,h):=
        \sum_{k\in\Z}|k|\widehat g(k)\overline{\widehat h(k)}
\end{equation}
and
\begin{equation}\label{eq:B-form}
        \BN(g,h):=
        \frac1{2\pi}\int_0^{2\pi}g(\theta)\overline{h(\theta)}w_N(\theta)\,d\theta.
\end{equation}
With the convention \eqref{eq:fourier-convention},
\begin{equation}\label{eq:E-Lambda}
        \EN(g,h)=\frac1{2\pi}\int_0^{2\pi}(\Lam g)\overline h\,d\theta.
\end{equation}

\begin{proposition}\label{prop:conformal-pullback}
Let $u\in H^1(\OmN)$ be harmonic, set $U=u\circ f_N$, and let $g:=\operatorname{Tr}_{\T}U$.
Then
\begin{equation}\label{eq:dirichlet-conformal}
        \int_{\OmN}|\nabla u|^2\,dx=\int_{\D}|\nabla U|^2\,dx
\end{equation}
and
\begin{equation}\label{eq:boundary-pushforward}
        \int_{\partial\OmN}|u|^2\,ds=
        \int_0^{2\pi}|g(\theta)|^2w_N(\theta)\,d\theta.
\end{equation}
Consequently, the Steklov equation on $\OmN$ is equivalent to the weighted disk problem
\begin{equation}\label{eq:weighted-weak-problem}
        \EN(g,h)=\sigma\BN(g,h),\qquad h\in H^{1/2}(\T),
\end{equation}
or, in boundary distribution form,
\begin{equation}\label{eq:weighted-dtn-equation}
        \Lam g=\sigma w_N g.
\end{equation}
Eigenfunctions with positive eigenvalue have weighted mean zero:
\begin{equation}\label{eq:weighted-mean-zero}
        \BN(g,1)=0.
\end{equation}
\end{proposition}

\begin{proof}
The energy identity \eqref{eq:dirichlet-conformal} follows from \Cref{lem:homogeneous-dirichlet-isometry}.  The trace identity \eqref{eq:polygon-to-disk-trace} and the unitary map \eqref{eq:L2-boundary-unitary} give \eqref{eq:boundary-pushforward}.

Assume first that $u$ is a Steklov eigenfunction, and let $h\in H^{1/2}(\T)$.  Let $V$ be the harmonic extension of $h$ and set $v=V\circ f_N^{-1}$.  By \Cref{prop:sobolev-pullback-trace}, $v\in H^1(\OmN)$ has trace $h\circ\gamma_N^{-1}$.  Pulling back \eqref{eq:steklov-weak} gives
\[
        \int_{\D}\nabla U\cdot\overline{\nabla V}\,dx
        =\sigma\int_0^{2\pi}g(\theta)\overline{h(\theta)}w_N(\theta)\,d\theta.
\]
For the harmonic extensions $U$ and $V$ of $g$ and $h$, Fourier expansion followed by density gives
\begin{equation}\label{eq:disk-energy-boundary-form}
        \frac1{2\pi}\int_{\D}\nabla U\cdot\overline{\nabla V}\,dx
        =\EN(g,h).
\end{equation}
Dividing by $2\pi$ proves \eqref{eq:weighted-weak-problem}.

Conversely, suppose that $g\in H^{1/2}(\T)$ satisfies \eqref{eq:weighted-weak-problem}.  Let $U$ be its harmonic extension and set $u=U\circ f_N^{-1}$.  By \Cref{prop:sobolev-pullback-trace}, $u\in H^1(\OmN)$ has trace $g\circ\gamma_N^{-1}$.  For $v\in H^1(\OmN)$, \Cref{prop:sobolev-pullback-trace} gives
\[
        h:=(\operatorname{Tr}_{\partial\OmN}v)\circ\gamma_N
        =\operatorname{Tr}_{\T}(v\circ f_N)\in H^{1/2}(\T).
\]
Let $V$ be the harmonic extension of $h$ and set $\widetilde v=V\circ f_N^{-1}$.  By the same proposition, $v$ and $\widetilde v$ have the same Sobolev trace, so $v-\widetilde v\in H^1_0(\OmN)$.  Harmonicity of $u$ and \eqref{eq:weighted-weak-problem}, tested with this $h$, yield \eqref{eq:steklov-weak} for $\widetilde v$ and hence for $v$.  The two weak problems are therefore equivalent.  Equations \eqref{eq:E-Lambda} and \eqref{eq:B-form} identify \eqref{eq:weighted-weak-problem} with the distributional equation \eqref{eq:weighted-dtn-equation}.  Finally, testing \eqref{eq:weighted-weak-problem} with $h=1$ gives $0=\EN(g,1)=\sigma\BN(g,1)$ and hence \eqref{eq:weighted-mean-zero} for $\sigma>0$.
\end{proof}

\subsection{Admissibility of the weighted form}

\begin{lemma}\label{lem:weighted-form-admissibility}
For every $N\ge3$, the form $\BN$ extends continuously to
$H^{1/2}(\T)\times H^{1/2}(\T)$.
\end{lemma}

\begin{proof}
Choose $1<p<N/2$ and set $p'=p/(p-1)$.  By \Cref{lem:sc-normalization}, $w_N\in L^p(\T)$, and $H^{1/2}(\T)\hookrightarrow L^{2p'}(\T)$.  H\"older's inequality then gives $|\BN(g,h)|\le C_N\|g\|_{H^{1/2}(\T)}\|h\|_{H^{1/2}(\T)}$.
\end{proof}

% ============================================================
\section{Fourier residue blocks and reciprocal operators}
\label{sec:blocks}
% ============================================================

The Fourier support of the Schwarz--Christoffel weight makes the weighted boundary problem block diagonal.  We compute the normalized Fourier coefficients, assemble the residue-class variational problem, and represent each nonzero block by a compact reciprocal operator on an energy-normalized $\ell^2$ space.

\subsection{Fourier coefficients of the polygonal weight}

Let $0<\alpha<1/2$ and
\[
        W_\alpha(t):=|1-e^{it}|^{-2\alpha}.
\]
Let $c_j(\alpha)$ denote its Fourier coefficients under the convention \eqref{eq:fourier-convention}, and normalize by
\[
        v_j(\alpha):=\frac{c_j(\alpha)}{c_0(\alpha)},\qquad v_0(\alpha)=1.
\]
The perimeter normalization implies
\[
        C_Nc_0(1/N)=1,
\]
where $c_0(\alpha)$ is the average Fourier coefficient of $W_\alpha$.  Hence $v_j=c_j/c_0$ are the Fourier coefficients of the normalized polygonal weight, and $w_N$ has the Abel representation
\begin{equation}\label{eq:weight-Fourier}
        w_N(\theta)=\lim_{\rho\to1^-}
        \sum_{j\in\Z}\rho^{|j|}v_j(1/N)e^{ijN\theta}
\end{equation}
in $L^1$ and hence in distributions.

\begin{lemma}\label{lem:normalized-fourier-coefficients}
For $j\ge1$,
\begin{equation}\label{eq:vj-formula}
        v_j(\alpha)=
        \frac{\Gamma(1-\alpha)}{\Gamma(\alpha)}
        \frac{\Gamma(j+\alpha)}{\Gamma(j+1-\alpha)},
        \qquad v_{-j}(\alpha)=v_j(\alpha)>0.
\end{equation}
\end{lemma}

\begin{proof}
Consider the radial regularization
\[
        W_{\alpha,\rho}(t)=|1-\rho e^{it}|^{-2\alpha},\qquad 0<\rho<1.
\]
As $\rho\to1^-$, the functions $W_{\alpha,\rho}$ converge to $W_\alpha$ in $L^1(\T)$.  The convergence is uniform away from $t=0$, while near $0$ the estimate
\[
        |1-\rho e^{it}|^{-2\alpha}\le C_\alpha((1-\rho)^2+t^2)^{-\alpha}
\]
is uniformly integrable because $\alpha<1/2$.

Expanding the holomorphic and antiholomorphic factors gives
\[
        (1-\rho e^{it})^{-\alpha}=
        \sum_{a\ge0}\frac{(\alpha)_a}{a!}\rho^ae^{iat}.
\]
Therefore, for $j\ge0$, the $j$-th Fourier coefficient of $W_{\alpha,\rho}$ is
\[
        c_j(\alpha,\rho)=
        \sum_{n\ge0}
        \frac{(\alpha)_{n+j}}{(n+j)!}
        \frac{(\alpha)_n}{n!}\rho^{2n+j}.
\]
Letting $\rho\to1^-$ and applying Gauss summation yields
\[
        c_j(\alpha)=
        \frac{(\alpha)_j}{j!}
        {}_2F_1(\alpha+j,\alpha;j+1;1)
        =
        \frac{\Gamma(j+\alpha)\Gamma(1-2\alpha)}
        {\Gamma(\alpha)\Gamma(1-\alpha)\Gamma(j+1-\alpha)}.
\]
For $j=0$ this gives
\[
        c_0(\alpha)=\frac{\Gamma(1-2\alpha)}{\Gamma(1-\alpha)^2}.
\]
Dividing by $c_0$ gives \eqref{eq:vj-formula}.  Since $W_\alpha$ is even, the coefficients are even; all Gamma factors in \eqref{eq:vj-formula} are positive for $0<\alpha<1/2$.
\end{proof}

\subsection{Residue-class decomposition}

Recall the Fourier residue spaces $\HN{N}{r}$ from \eqref{eq:intro-block-space}.  Equivalently,
\begin{equation}\label{eq:block-space}
        g\in\HN{N}{r}
        \quad\iff\quad
        g=\sum_{m\in\Z}\widehat g(r+mN)e^{i(r+mN)\theta}
        \quad\text{in }H^{1/2}(\T).
\end{equation}
The decomposition
\begin{equation}\label{eq:block-decomposition}
        H^{1/2}(\T)=\bigoplus_{r=0}^{N-1}\HN{N}{r}
\end{equation}
is an orthogonal Hilbert-space decomposition for the standard Fourier $H^{1/2}$ norm.  Since the Fourier modes of $w_N$ are divisible by $N$, both $\EN$ and $\BN$ are block diagonal.  This is immediate for $\EN$.  For $\BN$, it follows first for trigonometric polynomials and then for all $H^{1/2}$ traces by \Cref{lem:weighted-form-admissibility}.  More precisely, if
\[
        P_Jg=\sum_{|n|\le J}x_ne^{i(r+nN)\theta},
\]
then the definition of $v_j(1/N)$ as the Fourier coefficients of $w_N$ gives directly
\[
        \BN(P_Jg,e^{i(r+mN)\theta})
        =\sum_{|n|\le J}v_{m-n}(1/N)x_n.
\]
For general $g\in\HN{N}{r}$, this identity is understood through the symmetric Fourier truncations $P_Jg\to g$ in $H^{1/2}(\T)$, using the continuity of $\BN$.  Thus no pointwise convergence of the Fourier series of $w_N$ is required.

For $r=1,\ldots,N-1$ define
\begin{equation}\label{eq:block-rayleigh-nonzero}
        \sig_1^{(r)}(N):=
        \inf_{0\ne g\in\HN{N}{r}}
        \frac{\EN(g,g)}{\BN(g,g)}.
\end{equation}
For the zero block define
\begin{equation}\label{eq:block-rayleigh-zero}
        \sig_1^{(0)}(N):=
        \inf_{\substack{0\ne g\in\HN{N}{0}\\ \BN(g,1)=0}}
        \frac{\EN(g,g)}{\BN(g,g)}.
\end{equation}

\begin{lemma}\label{lem:block-assembly}
For every $N\ge3$,
\begin{equation}\label{eq:block-assembly}
        \sigone(\OmN)=\min_{0\le r\le N-1}\sig_1^{(r)}(N).
\end{equation}
Moreover, the zero-block infimum is attained.  If
$\sig_1^{(0)}(N)=\sigone(\OmN)$, a zero-block minimizer is a genuine first Steklov eigenfunction.
\end{lemma}

\begin{proof}
Write $g=\sum_{r=0}^{N-1}g_r$, where $g_r\in\HN{N}{r}$.  Block diagonality gives
\[
        \EN(g,g)=\sum_r\EN(g_r,g_r),\qquad
        \BN(g,g)=\sum_r\BN(g_r,g_r).
\]
For $r\ne0$, $\BN(g_r,1)=0$ because $w_N$ has only zero-residue Fourier modes.  Thus the global constraint $\BN(g,1)=0$ acts only on $g_0$.  Moreover, $\BN(g_r,g_r)>0$ whenever $g_r\ne0$, since $w_N>0$ almost everywhere and a nonzero $H^{1/2}$ trace cannot vanish almost everywhere.  The Rayleigh quotient of $g$ is therefore a weighted average of the quotients of its nonzero components.  This gives $\sigone(\OmN)\ge\min_r\sig_1^{(r)}(N)$, while the reverse inequality follows by testing with any admissible block function.

For the zero block set
\[
        \mathcal V_{N,0}:=\{g\in\HN{N}{0}:\BN(g,1)=0\}.
\]
The boundedness of $\BN$ shows that this subspace is closed in $H^{1/2}$.  Write $g=a+\phi$, where $\phi$ has zero Fourier constant term.  The constraint gives
\[
        a=-\BN(\phi,1),
\]
because $\BN(1,1)=1$.  Since $\phi$ has no constant Fourier mode, its $H^{1/2}$ norm is controlled by the energy norm.  Hence
\[
        |a|\le C\|\phi\|_{H^{1/2}}
        \le C'\EN(\phi,\phi)^{1/2}.
\]
Thus the energy controls the full $H^{1/2}$ norm on $\mathcal V_{N,0}$.  Any minimizing sequence normalized by $\BN(g_j,g_j)=1$ is bounded in $H^{1/2}$.  After passing to a subsequence, the compact embedding $H^{1/2}(\T)\hookrightarrow L^q(\T)$ gives $g_j\rightharpoonup g$ in $H^{1/2}$ and $g_j\to g$ in $L^q$ for every finite $q$.  Since $w_N\in L^p$, H\"older's inequality yields
\[
        \BN(g_j,g_j)\to\BN(g,g),\qquad \BN(g_j,1)\to\BN(g,1).
\]
Weak lower semicontinuity of $\EN$ gives a zero-block minimizer.

Its Euler--Lagrange equation initially holds for variations in $\mathcal V_{N,0}$.  Every $h\in\HN{N}{0}$ can be written as
\[
        h=(h-\BN(h,1)1)+\BN(h,1)1,
\]
where the first term belongs to $\mathcal V_{N,0}$.  Since $\EN(g,1)=0$ and $\BN(g,1)=0$, the equation extends to all zero-block tests.  If the zero-block value is global, the remaining residue blocks are orthogonal to the zero block for both forms, so the minimizer satisfies the full weak Steklov equation.
\end{proof}

\subsection{Energy normalization and reciprocal operators}

For $1\le r\le N-1$ set
\begin{equation}\label{eq:dm-definition}
        d_m^{(N,r)}:=|r+mN|.
\end{equation}
Define
\begin{equation}\label{eq:U-map}
        (U_{N,r}y)(\theta):=
        \sum_{m\in\Z}\frac{y_m}{\sqrt{d_m^{(N,r)}}}
        e^{i(r+mN)\theta},\qquad y\in\ell^2(\Z).
\end{equation}

\begin{proposition}\label{prop:reciprocal-block}
For $1\le r\le N-1$, the map $U_{N,r}$ is an isometric identification of $\ell^2(\Z)$ with the $r$-block endowed with the energy norm:
\begin{equation}\label{eq:energy-normalization}
        \EN(U_{N,r}y,U_{N,r}z)=\langle y,z\rangle_{\ell^2}.
\end{equation}
The matrix
\begin{equation}\label{eq:A-N-r}
        (A_{N,r})_{m,n}:=
        \frac{v_{m-n}(1/N)}
        {\sqrt{|r+mN|\,|r+nN|}}
\end{equation}
extends to a compact self-adjoint positive semidefinite Hilbert--Schmidt operator on $\ell^2(\Z)$ with strictly positive entries, and
\begin{equation}\label{eq:B-A-identity}
        \BN(U_{N,r}y,U_{N,r}z)=\langle A_{N,r}y,z\rangle_{\ell^2}
\end{equation}
for all $y,z\in\ell^2(\Z)$.  Moreover,
\begin{equation}\label{eq:block-reciprocal}
        \sig_1^{(r)}(N)=\frac1{\lambda_{\max}(A_{N,r})}.
\end{equation}
\end{proposition}

\begin{proof}
For finitely supported $y,z$, \eqref{eq:energy-normalization} follows directly from \eqref{eq:E-form}.  The finite coefficient identity in Section~\ref{sec:blocks} gives
\[
        \BN(U_{N,r}y,U_{N,r}z)
        =\sum_{m,n\in\Z}
        \frac{v_{m-n}(1/N)y_n\overline{z_m}}
        {\sqrt{|r+nN|\,|r+mN|}},
\]
which is \eqref{eq:B-A-identity} on finitely supported vectors.

The identity $v_{-j}=v_j$ gives symmetry of the matrix, and \Cref{lem:normalized-fourier-coefficients} gives strict positivity of all its entries.  We next prove that it is Hilbert--Schmidt.  By the Gamma formula \eqref{eq:vj-formula} and Stirling's formula, for every fixed $\alpha\in(0,1/2)$,
\[
        \frac{\Gamma(j+\alpha)}{\Gamma(j+1-\alpha)}
        =j^{2\alpha-1}\bigl(1+O_\alpha(j^{-1})\bigr),
        \qquad j\to\infty.
\]
Hence, after enlarging the constant to cover the finitely many small indices,
\[
        |v_j(\alpha)|\le C_\alpha(1+|j|)^{2\alpha-1},\qquad j\in\Z.
\]
For $r=1,\ldots,N-1$,
\[
        |r+mN|\ge c_{N,r}(1+|m|).
\]
With $\alpha=1/N\le1/3$, we obtain
\[
\begin{aligned}
        \sum_{m,n\in\Z}|(A_{N,r})_{m,n}|^2
        &\le C\sum_{m,n\in\Z}
        \frac{(1+|m-n|)^{4\alpha-2}}
        {(1+|m|)(1+|n|)}.
\end{aligned}
\]
Writing $j=m-n$, we use the elementary convolution estimate
\begin{equation}\label{eq:convolution-bound}
        \sum_{n\in\Z}\frac1{(1+|n+j|)(1+|n|)}
        \le C\frac{\log(e+|j|)}{1+|j|},
        \qquad j\in\Z.
\end{equation}
Indeed, after symmetry reduction to $j\ge0$, first split off the tail $|n|>2j+2$, whose contribution is $O((1+j)^{-1})$.  In the remaining finite region, separate the neighborhoods $|n|\le j/2$ and $|n+j|\le j/2$.  Each contributes at most $C(1+j)^{-1}\log(e+j)$.  On the rest of the finite region both factors are bounded below by $c(1+j)$ and there are $O(j)$ terms, giving another $O((1+j)^{-1})$ contribution.  This proves \eqref{eq:convolution-bound}.  Consequently,
\[
        \|A_{N,r}\|_{\mathrm{HS}}^2
        \le C\sum_{j\in\Z}(1+|j|)^{4\alpha-3}\log(e+|j|)<\infty,
\]
since $\alpha\le1/3$.  Thus $A_{N,r}$ extends to a bounded Hilbert--Schmidt operator and is compact and self-adjoint.

Because $r\ne0$, the energy norm controls the full $H^{1/2}$ norm on the block.  Density therefore extends \eqref{eq:energy-normalization} and \eqref{eq:B-A-identity} to all $y,z\in\ell^2(\Z)$.  Conversely, for $g\in\HN{N}{r}$, the sequence
\[
        y_m:=\sqrt{|r+mN|}\,\widehat g(r+mN)
\]
belongs to $\ell^2(\Z)$ and satisfies $g=U_{N,r}y$, so $U_{N,r}$ is onto.

The identity \eqref{eq:B-A-identity} gives
\[
        \langle A_{N,r}y,y\rangle_{\ell^2}
        =\BN(U_{N,r}y,U_{N,r}y)
        =\frac1{2\pi}\int_0^{2\pi}|U_{N,r}y|^2w_N\,d\theta\ge0.
\]
For $y\ne0$, this boundary mass is strictly positive because $w_N>0$ almost everywhere.  Finally,
\[
        \frac{\EN(U_{N,r}y,U_{N,r}y)}{\BN(U_{N,r}y,U_{N,r}y)}
        =\frac{\|y\|_{\ell^2}^2}{\langle A_{N,r}y,y\rangle_{\ell^2}}.
\]
Compact self-adjointness gives $\lambda_{\max}(A_{N,r})>0$, and the preceding identity proves \eqref{eq:block-reciprocal}.
\end{proof}

\begin{corollary}\label{cor:minimizing-block-eigenfunction}
If $\sig_1^{(r)}(N)=\sigone(\OmN)$ for some $0\le r\le N-1$, then the corresponding closed constrained block contains a genuine first Steklov eigenfunction.
\end{corollary}

\begin{proof}
For $r=0$, this is the final assertion of \Cref{lem:block-assembly}.  Let $r\ne0$ and choose a top eigenvector $y$ of $A_{N,r}$.  By \Cref{prop:reciprocal-block}, $g=U_{N,r}y$ attains the block quotient and satisfies the Euler--Lagrange equation for all tests in $\HN{N}{r}$.  The weighted mean constraint is automatic, and block diagonality makes both forms vanish against tests in every other residue class.  Hence $g$ satisfies the full weak Steklov equation at the global value.
\end{proof}

% ============================================================
\section{Structure of the first Steklov level}
\label{sec:first-level-structure}
% ============================================================

We now prove \Cref{thm:first-level-structure}.  We first use boundary nodal topology to identify the two residue blocks that carry the first Steklov level.  Once this localization is established, the same positive reciprocal operator that drives monotonicity also yields Perron simplicity and the structure of the first eigenspace.  The topological input is the Euler-type boundary estimate
\[
        N_\partial(u)\le 2\bigl(K(u)-1\bigr),
\]
where $K(u)$ is the number of interior nodal domains and $N_\partial(u)$ is the number of boundary nodal domains.  The classical first-level nodal-line mechanism goes back to Kuttler--Sigillito \cite{KuttlerSigillito1969}.  At the regularity level needed here, the Courant nodal-domain theorem is available on bounded Lipschitz domains; see Hassannezhad--Sher \cite[Theorem~2.1]{HassannezhadSher2026}.  The Euler step is purely topological once the nodal set is a finite properly embedded forest in a disk: \Cref{lem:topological-euler-boundary-count} states the corresponding graph count abstractly.  This is the mechanism behind the boundary-graph estimate of Alessandrini--Magnanini \cite[Definition~3.1 and Lemma~3.4]{AlessandriniMagnanini1994}; see also \cite[Section~6.1]{GirouardPolterovich2017}.  Appendix~\ref{app:nodal-input} verifies the required finite-forest hypotheses for the regular polygons considered here and applies the abstract count in \Cref{cor:polygonal-euler-boundary-count}.  Combined with Courant's bound $K(u)\le2$ at the first positive level, this gives exactly two boundary nodal domains.  This is the only nodal input used below.  Quotient periodicity excludes the zero block, and a low-mode obstruction excludes the middle blocks.

\subsection{Boundary nodal input}

\begin{definition}\label{def:boundary-nodal-domains}
Let $\Gamma$ be a Jordan curve and let $\varphi\in C(\Gamma;\R)$ be nonzero.  The boundary nodal domains of $\varphi$ are the connected components of the two open sets
\[
        \{x\in\Gamma:\varphi(x)>0\},
        \qquad
        \{x\in\Gamma:\varphi(x)<0\}.
\]
Both positive and negative components are counted.  In the polygonal setting, they are finite open arcs by \Cref{cor:finite-boundary-sign-components}.
\end{definition}

\begin{proposition}\label{prop:first-level-boundary-nodal}
Let $u$ be a real first nonzero Steklov eigenfunction on the regular polygon $\OmN$.
Then $u|_{\partial\OmN}$ has exactly two boundary nodal domains: one positive and one negative.
\end{proposition}

\begin{proof}
The Courant nodal-domain theorem used here is valid on bounded Lipschitz domains; see Hassannezhad--Sher \cite[Theorem~2.1]{HassannezhadSher2026}, with Kuttler--Sigillito \cite{KuttlerSigillito1969} as the classical first-level source.  For a real vector $u$ in the first positive eigenspace, choose an orthonormal eigenbasis whose first vector is constant and whose second vector is the normalization of $u$.  Courant's theorem gives
\[
        K(u)\le2.
\]
Applying the topological Euler count through \Cref{cor:polygonal-euler-boundary-count} gives
\[
        N_\partial(u)\le2\bigl(K(u)-1\bigr)\le2.
\]
On the other hand, since the eigenvalue is positive,
\[
        \int_{\partial\OmN}u\,ds=0.
\]
The nonzero continuous trace therefore takes both signs on the connected boundary.  It has at least one positive and one negative boundary nodal domain, so the upper bound is attained.
\end{proof}

\subsection{Critical-block selection}

\Cref{prop:first-level-boundary-nodal} gives a topological constraint on the first boundary trace. Since the boundary homeomorphism $\gamma_N:\T\to\partial\OmN$ preserves sign components, every real pulled-back first Steklov trace has exactly two nodal components on $\T$, one positive and one negative. We now combine this constraint with the Fourier support of the residue blocks. A low-mode obstruction excludes the middle classes $2\le r\le N-2$, while the $2\pi/N$-periodicity of the zero block excludes $r=0$.

\begin{lemma}\label{lem:low-mode}
Let $g\in C(\T;\R)$ be nonzero.  If
\begin{equation}\label{eq:low-modes}
        \widehat g(0)=\widehat g(1)=\widehat g(-1)=0,
\end{equation}
then the positive and negative sign sets of $g$ have strictly more than two connected components in total.
\end{lemma}

\begin{proof}
Since $\widehat g(0)=0$ and $g\not\equiv0$, the function $g$ changes sign.  Suppose that the two sign sets have at most two connected components in total.  Each sign set is then a single open arc.  Write
\[
        P=\{\theta:g(\theta)>0\}=(c-h,c+h),\qquad 0<h<\pi,
\]
and set $T(\theta)=\cos(\theta-c)-\cos h$.  Then $T>0$ on $P$ and $T<0$ on $\T\setminus\overline P$.  Hence $gT\ge0$ everywhere and $gT>0$ on a nonempty open set, so
\[
        \int_0^{2\pi}g(\theta)T(\theta)\,d\theta>0.
\]
Since $T$ is a real linear combination of the modes $0,\pm1$, \eqref{eq:low-modes} makes the same integral zero.  This contradiction proves the lemma.
\end{proof}

\begin{lemma}\label{lem:critical-block-exclusion}
If a residue block $\HN{N}{r}$ carries the first Steklov eigenvalue $\sigone(\OmN)$, then
\[
        r=1\qquad\text{or}\qquad r=N-1.
\]
Equivalently, neither the zero block nor any middle block $\HN{N}{r}$, $2\le r\le N-2$, can carry the first Steklov level.
\end{lemma}

\begin{proof}
We first exclude the middle blocks.  Suppose that $N\ge4$ and that $\HN{N}{r}$ carries the first eigenvalue for some $2\le r\le N-2$.  By \Cref{cor:minimizing-block-eigenfunction}, there exists a nonzero complex first eigenfunction $g\in\HN{N}{r}$.  Since the weighted equation has real coefficients, one of $\Re g$ and $\Im g$ is a nonzero real weak solution; denote it by $\phi$.  By \Cref{prop:conformal-pullback}, $\phi$ is the pullback of a real first Steklov eigenfunction on $\OmN$, and \Cref{lem:finite-boundary-sign-components} gives a continuous representative on $\T$.  Passing to real or imaginary parts adds only conjugate frequencies; hence
\[
        \operatorname{supp}\widehat\phi
        \subset (r+N\Z)\cup(-r+N\Z).
\]
For $2\le r\le N-2$ this set contains none of $0,\pm1$, and therefore
\[
        \widehat\phi(0)=\widehat\phi(1)=\widehat\phi(-1)=0.
\]
Lemma~\ref{lem:low-mode} then gives more than two nodal components on $\T$, whereas the boundary-nodal transfer gives exactly two.  This is a contradiction.  For $N=3$, there are no middle residue classes.

It remains to exclude $r=0$.  If the zero block carried the first eigenvalue, \Cref{cor:minimizing-block-eigenfunction} would yield a nonzero first-eigenfunction trace in $\HN{N}{0}$ with weighted mean zero.  Taking a nonzero real or imaginary part, we obtain a real first-eigenfunction trace $g\in\HN{N}{0}$ with $\BN(g,1)=0$.  Both $g$ and $w_N$ are $2\pi/N$-periodic.  The descended trace on $\T_N=\R/(2\pi/N)\Z$ therefore has zero weighted mean and, since $w_N>0$ almost everywhere, must change sign.  It has at least one positive and one negative component on $\T_N$.  Each proper component lifts under the degree-$N$ covering $\T\to\T_N$ to $N$ disjoint components, so $g$ has at least $2N\ge6$ nodal components on $\T$.  This again contradicts the boundary-nodal transfer.  Hence only $r=1$ and $r=N-1$ can carry the first level.
\end{proof}

\medskip
\noindent We now complete the critical-block localization in \Cref{thm:first-level-structure}.  By \Cref{lem:block-assembly}, $\sigone(\OmN)$ is the minimum of the block values.  By \Cref{lem:critical-block-exclusion}, only the two critical blocks $r=1$ and $r=N-1$ remain.  Complex conjugation maps $\HN{N}{1}$ anti-linearly and isometrically onto $\HN{N}{N-1}$ and, since $w_N$ is real-valued, preserves both $\EN$ and $\BN$.  Hence the two critical Rayleigh values coincide, and therefore
\[
        \sigone(\OmN)=\sig_1^{(1)}(N)=\sig_1^{(N-1)}(N).
\]
If a noncritical block satisfied $\sig_1^{(r)}(N)=\sigone(\OmN)$, \Cref{cor:minimizing-block-eigenfunction} would produce a first Steklov eigenfunction in that block, contradicting \Cref{lem:critical-block-exclusion}.  Hence all noncritical inequalities are strict.

Finally, let $g$ be a pulled-back first trace and write $g=\sum_{r=0}^{N-1}g_r$, with $g_r\in\HN{N}{r}$.  By block diagonality, testing the weak equation against $h\in\HN{N}{r}$ gives
\[
        \EN(g_r,h)=\sigone(\OmN)\BN(g_r,h),\qquad h\in\HN{N}{r}.
\]
Moreover, $\BN(g_0,1)=\BN(g,1)=0$.  Every nonzero block component therefore carries the first level, and \Cref{lem:critical-block-exclusion} implies $g_r=0$ unless $r=1$ or $r=N-1$.  Therefore
\[
        g\in\HN{N}{1}\oplus\HN{N}{N-1}.
\]
This completes the critical-block localization.

% ============================================================

\subsection{Perron structure of the critical block}

With the critical-block localization established, set
\begin{equation}\label{eq:critical-operator-definition}
        A_N:=A_{N,1},\qquad \lambda_N:=\lambda_{\max}(A_N).
\end{equation}
Together with \Cref{prop:reciprocal-block}, the critical-block identity gives
\begin{equation}\label{eq:sigma-lambda-reciprocal}
        \sigone(\OmN)=\frac1{\lambda_N}.
\end{equation}
By \Cref{prop:reciprocal-block}, $A_N$ is compact, self-adjoint, positive semidefinite, and Hilbert--Schmidt, with
\begin{equation}\label{eq:critical-entries-positive}
        (A_N)_{m,n}>0,
        \qquad m,n\in\Z.
\end{equation}
Since $(A_N)_{0,0}=1$, we also have $\lambda_N\ge1$.

\begin{lemma}\label{lem:perron-simplicity}
The maximal eigenvalue $\lambda_N$ of $A_N$ is simple.  Moreover, there is a unique normalized top eigenvector $p_N$ such that
\[
        A_Np_N=\lambda_Np_N,
        \qquad
        \|p_N\|_{\ell^2}=1,
        \qquad
        (p_N)_m>0\quad(m\in\Z).
\]
\end{lemma}

\begin{proof}
Compact self-adjointness and $\lambda_N\ge1$ provide a nonzero top eigenvector $z$.  The Hilbert--Schmidt property gives absolute convergence of the quadratic forms, and entrywise nonnegativity yields
\[
        \lambda_N\|z\|^2
        =\operatorname{Re}\langle A_Nz,z\rangle
        \le \langle A_N|z|,|z|\rangle
        \le \lambda_N\|z\|^2.
\]
Both inequalities are therefore equalities, and $|z|$ is again a $\lambda_N$-eigenvector.  By \eqref{eq:critical-entries-positive}, $A_N|z|>0$ coordinatewise; hence $|z|=\lambda_N^{-1}A_N|z|>0$.  Normalizing $|z|$ gives a strictly positive top eigenvector $p_N$.

Let $y$ be a nonzero real top eigenvector.  The same equality argument shows that $|y|$ is a strictly positive top eigenvector and that
\[
        0=\sum_{m,n\in\Z}(A_N)_{m,n}
        \bigl(|y_m|\,|y_n|-y_my_n\bigr).
\]
Each summand is nonnegative.  Since all entries in \eqref{eq:critical-entries-positive} are strictly positive, $y$ cannot have coordinates of opposite signs.  Thus every nonzero real top eigenvector has a fixed strict sign.  If the real top eigenspace had dimension at least two, it would contain a nonzero vector orthogonal to $p_N$, which is impossible for a vector of fixed strict sign.  Hence the real top eigenspace is one-dimensional.  Because $A_N$ has real entries, the real and imaginary parts of a complex top eigenvector are real top eigenvectors.  The complex top eigenspace is therefore one-dimensional, and the normalized strictly positive representative is unique.
\end{proof}

Using this normalized vector, define the distinguished critical trace
\begin{equation}\label{eq:critical-perron-trace}
        \psi_N:=U_{N,1}p_N
        =\sum_{m\in\Z}\frac{(p_N)_m}{|1+mN|^{1/2}}e^{i(1+mN)\theta}
        \in\HN{N}{1}.
\end{equation}

\subsection{First eigenspace structure}

We now determine the first eigenspace and complete the proof of \Cref{thm:first-level-structure}.  We first identify the first-eigenvalue solutions in the critical block $\HN{N}{1}$.  Write $g=U_{N,1}y$.  For every $z\in\ell^2(\Z)$, \Cref{prop:reciprocal-block} gives
\[
        \EN(g,U_{N,1}z)=\langle y,z\rangle,
        \qquad
        \BN(g,U_{N,1}z)=\langle A_Ny,z\rangle.
\]
Since $\sigone(\OmN)=1/\lambda_N$ by \eqref{eq:sigma-lambda-reciprocal}, the weak equation on $\HN{N}{1}$ is equivalent to
\[
        A_Ny=\lambda_Ny.
\]
By \Cref{lem:perron-simplicity}, its solution space is therefore the complex line $\C\psi_N$.

Complex conjugation maps $\HN{N}{1}$ anti-linearly and isometrically onto $\HN{N}{N-1}$.  Indeed,
\[
        \overline{e^{i(1+mN)\theta}}
        =e^{i(N-1+(-m-1)N)\theta}.
\]
Since $w_N$ is real-valued, conjugation preserves both $\EN$ and $\BN$.  Hence the first-eigenvalue solution space in $\HN{N}{N-1}$ is $\C\overline{\psi_N}$.

Let $E_{1,\C}^{\partial}(N)$ denote the complex first boundary trace space after Schwarz--Christoffel pullback.  The two lines above are contained in this space.  Conversely, if $g\in E_{1,\C}^{\partial}(N)$, then its real and imaginary parts are real first-eigenvalue solutions.  By the localization established above and recombining, we obtain
\[
        g=g_1+g_{N-1},
        \qquad
        g_1\in\HN{N}{1},\quad g_{N-1}\in\HN{N}{N-1}.
\]
By block diagonality, every nonzero component is itself a first-eigenvalue solution in its critical block.  Therefore
\begin{equation}\label{eq:complex-first-trace-space}
        E_{1,\C}^{\partial}(N)
        =\C\psi_N\oplus\C\overline{\psi_N}.
\end{equation}
The sum is direct because the two residue blocks are distinct for $N\ge3$.

The real first trace space is the fixed-point set of complex conjugation in \eqref{eq:complex-first-trace-space}; hence it is
\[
        \{a\psi_N+\overline a\,\overline{\psi_N}:a\in\C\}
        =\Span_\R\{\operatorname{Re}\psi_N,\operatorname{Im}\psi_N\}.
\]
The map
\[
        a\longmapsto a\psi_N+\overline a\,\overline{\psi_N}
\]
is injective as a real-linear map because $\HN{N}{1}\oplus\HN{N}{N-1}$ is a direct sum.  Thus the pulled-back real first trace space has dimension two.  Harmonic extension to the disk and inverse Schwarz--Christoffel pullback are injective, so the real first Steklov eigenspace on $\OmN$ also has dimension two.  This completes the proof of \Cref{thm:first-level-structure}.\qed

% ============================================================
\section{Perron comparison and global monotonicity}
\label{sec:perron-monotonicity}
% ============================================================

With $A_N$ and $\lambda_N$ defined in \eqref{eq:critical-operator-definition}, global monotonicity reduces to comparing $A_N$ and $A_{N+1}$ on the common space $\ell^2(\Z)$.  The matrix indices $m,n$ are fixed, whereas the corresponding Fourier frequencies are $1+mN$ and $1+m(N+1)$.

\subsection{Monotonicity of Fourier coefficients}

\begin{lemma}\label{lem:coefficient-monotonicity}
For every $j\in\Z\setminus\{0\}$, the map $\alpha\mapsto v_j(\alpha)$ is strictly increasing on $(0,1/2)$.
\end{lemma}

\begin{proof}
For $j\ge1$, differentiating the logarithm of \eqref{eq:vj-formula} gives
\[
\begin{aligned}
\frac{d}{d\alpha}\log v_j(\alpha)
&= -\psi(1-\alpha)-\psi(\alpha)+\psi(j+\alpha)+\psi(j+1-\alpha)\\
&=[\psi(j+\alpha)-\psi(\alpha)]+[\psi(j+1-\alpha)-\psi(1-\alpha)]\\
&=\sum_{q=0}^{j-1}
\left(\frac1{q+\alpha}+\frac1{q+1-\alpha}\right)>0,
\end{aligned}
\]
where $\psi$ is the digamma function and $\psi(x+j)-\psi(x)=\sum_{q=0}^{j-1}(q+x)^{-1}$.  Thus $v_j$ is strictly increasing for $j\ge1$.  Since $v_{-j}=v_j$, the same conclusion holds for negative $j$.
\end{proof}

\subsection{Entrywise comparison}

\begin{lemma}\label{lem:entrywise-comparison}
For every $N\ge3$ and all $m,n\in\Z$,
\begin{equation}\label{eq:entrywise-comparison}
        0<(A_{N+1})_{m,n}\le (A_N)_{m,n}.
\end{equation}
The inequality is strict unless $(m,n)=(0,0)$.
\end{lemma}

\begin{proof}
Write
\[
        d_m(N):=|1+mN|.
\]
We have $d_0(N)=1$.  If $m>0$, then $d_m(N)=1+mN$; if $m=-q<0$ with $q\ge1$, then $d_m(N)=qN-1$ because $N\ge3$.  In either case, $d_m(N+1)>d_m(N)$ for $m\ne0$.

Since $1/(N+1)<1/N$, \Cref{lem:coefficient-monotonicity} gives
\[
        v_j(1/(N+1))<v_j(1/N),\qquad j\ne0,
\]
while $v_0=1$ for all $N$.  If $m\ne n$, the numerator decreases strictly and the denominator does not decrease.  If $m=n\ne0$, the numerator remains equal to $1$ and the denominator increases strictly.  At $(m,n)=(0,0)$, both entries equal $1$.  Positivity follows from \Cref{lem:normalized-fourier-coefficients}.
\end{proof}

\subsection{Strict Perron comparison}

\begin{lemma}\label{lem:strict-perron-comparison}
Let $A$ and $B$ be compact self-adjoint positive semidefinite Hilbert--Schmidt operators on $\ell^2(\Z)$ with real matrix entries in the standard basis.
Assume that the entries of $B$ are strictly positive, that
\[
        0<B_{m,n}\le A_{m,n}\qquad(m,n\in\Z),
\]
and that $A_{p,q}>B_{p,q}$ for at least one pair $(p,q)$.
Then
\begin{equation}\label{eq:strict-perron-comparison}
        \lambda_{\max}(A)>\lambda_{\max}(B).
\end{equation}
\end{lemma}

\begin{proof}
Since $B_{0,0}>0$, we have $\lambda_{\max}(B)>0$.  The modulus argument in the proof of \Cref{lem:perron-simplicity}, which uses only compact self-adjointness, nonnegativity, and strict positivity of the matrix entries, then gives a top eigenvector $x>0$ for $B$.  Set $D=A-B$.  Its entries are nonnegative and at least one is strictly positive, so $\langle Dx,x\rangle>0$.  Therefore
\[
        \lambda_{\max}(A)
        \ge \frac{\langle Ax,x\rangle}{\|x\|^2}
        =\lambda_{\max}(B)+\frac{\langle Dx,x\rangle}{\|x\|^2}
        >\lambda_{\max}(B).
\]
\end{proof}

\subsection{Proof of global monotonicity}

\begin{proof}[Proof of \Cref{thm:global-monotonicity}]
Apply \Cref{lem:strict-perron-comparison} with $A=A_N$ and $B=A_{N+1}$.  Lemma~\ref{lem:entrywise-comparison} verifies the hypotheses; for example, the entrywise inequality is strict at $(m,n)=(1,1)$.  Hence
\[
        \lambda_N>\lambda_{N+1}.
\]
Equation \eqref{eq:sigma-lambda-reciprocal} then gives
\[
        \sigone(\Omega_{N+1})=\frac1{\lambda_{N+1}}
        >\frac1{\lambda_N}=\sigone(\Omega_N).
\]
\end{proof}

\section{Asymptotics in the critical block}
\label{sec:schur-asymptotics}
% ============================================================

By \eqref{eq:sigma-lambda-reciprocal}, $\sigone(\OmN)$ is the reciprocal of the top eigenvalue of $A_N$.  We obtain the leading terms of this eigenvalue by separating the principal Fourier mode and applying a Schur reduction.  Throughout this section, set
\[
        \alpha:=\frac1N,
        \qquad
        d_m:=|1+mN|=|1+m/\alpha|,
        \qquad m\in\Z.
\]
All estimates are taken as $N\to\infty$, equivalently $\alpha\downarrow0$.

\subsection{Splitting off the principal mode}

Let $e_0\in\ell^2(\Z)$ denote the basis vector corresponding to Fourier frequency $1$.
Write
\begin{equation}\label{eq:l2-splitting}
        \ell^2(\Z)=\Span\{e_0\}\oplus \ell^2(\Z_*) ,
        \qquad \Z_*:=\Z\setminus\{0\}.
\end{equation}
With respect to this decomposition, the critical reciprocal operator has the block form
\begin{equation}\label{eq:schur-block-matrix}
        A_N=
        \begin{pmatrix}
                1 & b_N^*\\
                b_N & K_N
        \end{pmatrix}.
\end{equation}
Here
\begin{equation}\label{eq:b-K-entries}
        (b_N)_m=\frac{v_m(\alpha)}{\sqrt{d_m}},
        \qquad
        (K_N)_{m,n}=\frac{v_{m-n}(\alpha)}{\sqrt{d_m d_n}},
        \qquad m,n\in\Z_*.
\end{equation}
Set also
\begin{equation}\label{eq:beta-kappa}
        \beta_N:=\|b_N\|_{\ell^2(\Z_*)}^2,
        \qquad
        \kappa_N:=\|K_N\|_{\ell^2(\Z_*)\to\ell^2(\Z_*)}.
\end{equation}
By \Cref{lem:tail-bounds},
\begin{equation}\label{eq:beta-kappa-orders}
        \beta_N=O(N^{-3}),
        \qquad
        \kappa_N=O(N^{-1}).
\end{equation}
Hence $\kappa_N<1/2$ for all sufficiently large $N$.

\subsection{The Schur scalar equation}

\begin{proposition}\label{prop:schur-equation}
There exist $N_0$ and $\varepsilon_0\in(0,1/4)$ such that, for every $N\ge N_0$, the maximal eigenvalue $\lambda_N$ is the unique solution $\lambda\in(1-\varepsilon_0,1+\varepsilon_0)$ of
\begin{equation}\label{eq:schur-equation}
        \lambda
        =1+\left\langle (\lambda I-K_N)^{-1}b_N,b_N\right\rangle_{\ell^2(\Z_*)}.
\end{equation}
Equivalently, with $\delta_N:=\lambda_N-1$,
\begin{equation}\label{eq:lambda-delta}
        \delta_N
        =\left\langle ((1+\delta_N)I-K_N)^{-1}b_N,b_N\right\rangle_{\ell^2(\Z_*)}.
\end{equation}
\end{proposition}

\begin{proof}
Let $(x,y)\in\C\oplus\ell^2(\Z_*)$ be an eigenvector of $A_N$ with eigenvalue $\lambda$.  The block equations are
\begin{equation}\label{eq:block-eigen-equations}
        x+\langle y,b_N\rangle=\lambda x,
        \qquad
        b_Nx+K_Ny=\lambda y.
\end{equation}
If $\lambda$ lies in a fixed small neighborhood of $1$, then $\lambda I-K_N$ is invertible for all large $N$ because $\kappa_N<1/2$.  Moreover, $x\ne0$: otherwise $K_Ny=\lambda y$, which implies $|\lambda|\le\kappa_N<1/2$.  Rescaling the eigenvector to $x=1$, the second equation gives
\[
        y=(\lambda I-K_N)^{-1}b_N.
\]
Substituting this expression into the first equation gives \eqref{eq:schur-equation}.

Conversely, if \eqref{eq:schur-equation} holds and $y=(\lambda I-K_N)^{-1}b_N$, then $(1,y)$ is an eigenvector of $A_N$.
The function
\[
        F_N(\lambda):=\lambda-1-
        \left\langle (\lambda I-K_N)^{-1}b_N,b_N\right\rangle
\]
is strictly increasing on $(\kappa_N,\infty)$, since
\[
        F_N'(\lambda)=1+
        \left\langle (\lambda I-K_N)^{-2}b_N,b_N\right\rangle>0.
\]
We next identify the Schur root with the maximal eigenvalue.  The Rayleigh quotient of $A_N$ at $e_0$ equals $1$, so $\lambda_N\ge1$.  Since $\kappa_N<1/2$ for all sufficiently large $N$, we have $\lambda_N>\kappa_N$.  If $(x,y)$ is a top eigenvector, the second equation in \eqref{eq:block-eigen-equations} implies $x\ne0$; otherwise $K_Ny=\lambda_Ny$ would give $\lambda_N\le\kappa_N$.  Rescaling to $x=1$, the preceding calculation shows that $\lambda_N$ satisfies \eqref{eq:schur-equation}.

Since $F_N$ is strictly increasing on $(\kappa_N,\infty)$, the Schur equation has at most one solution there; this solution is $\lambda_N$.  The equation also yields the estimate used below.  Since $\lambda_N\ge1$ and $\kappa_N<1/2$,
\[
        0\le \lambda_N-1
        =\left\langle(\lambda_N I-K_N)^{-1}b_N,b_N\right\rangle
        \le \frac{\|b_N\|^2}{\lambda_N-\kappa_N}
        \le 2\beta_N.
\]
Because $\beta_N=O(N^{-3})$, a sufficiently large $N_0$ ensures that $\lambda_N\in(1-\varepsilon_0,1+\varepsilon_0)$.  Since $\kappa_N<1/2$, the operator $K_N$ has no spectrum in this interval.  Thus $\lambda_N$ is the unique Schur root in the stated neighborhood of $1$.
\end{proof}

\subsection{Moment expansion}

Define the Schur moments
\begin{equation}\label{eq:schur-moments}
        M_q(N):=\langle K_N^q b_N,b_N\rangle_{\ell^2(\Z_*)},
        \qquad q=0,1,2,\ldots.
\end{equation}
Since $K_N$ is positive semidefinite, every $M_q(N)$ is nonnegative.

\begin{lemma}\label{lem:moment-expansion-main}
Let $J\ge0$ be fixed.  For $\lambda>\kappa_N$,
\begin{equation}\label{eq:resolvent-moment-expansion-main}
        \left\langle (\lambda I-K_N)^{-1}b_N,b_N\right\rangle
        =\sum_{q=0}^{J}\lambda^{-q-1}M_q(N)
          +R_J(N,\lambda),
\end{equation}
where
\begin{equation}\label{eq:moment-remainder-bound-main}
        |R_J(N,\lambda)|
        \le
        \beta_N\frac{\kappa_N^{J+1}}{\lambda^{J+1}(\lambda-\kappa_N)}.
\end{equation}
Consequently,
\begin{equation}\label{eq:delta-M0-M1}
        \delta_N=M_0(N)+M_1(N)+M_2(N)+O(N^{-6}).
\end{equation}
\end{lemma}

\begin{proof}
Since $\lambda>\kappa_N=\|K_N\|$, we have $\|\lambda^{-1}K_N\|<1$, and the Neumann series gives
\[
        (\lambda I-K_N)^{-1}
        =\sum_{q=0}^{J}\lambda^{-q-1}K_N^q+T_J(\lambda),
\]
with
\[
        \|T_J(\lambda)\|
        \le \sum_{q\ge J+1}\lambda^{-q-1}\kappa_N^q
        =\frac{\kappa_N^{J+1}}{\lambda^{J+1}(\lambda-\kappa_N)}.
\]
Pairing with $b_N$ proves \eqref{eq:resolvent-moment-expansion-main} and \eqref{eq:moment-remainder-bound-main}.

Apply the expansion with $J=2$ and $\lambda=\lambda_N$.  By \Cref{prop:schur-equation}, the left-hand side equals $\delta_N$.  Equation \eqref{eq:beta-kappa-orders} and the bound $\lambda_N-1=O(\beta_N)$ from \Cref{prop:schur-equation} give
\[
        R_2(N,\lambda_N)=O(\beta_N\kappa_N^3)=O(N^{-6}).
\]
The following estimates control the prefactors:
\[
\begin{aligned}
        M_0(N)&=\langle b_N,b_N\rangle=\beta_N=O(N^{-3}),\\
        M_1(N)&\le \|K_N\|\,\|b_N\|^2=O(N^{-4}),\\
        M_2(N)&\le \|K_N\|^2\|b_N\|^2=O(N^{-5}).
\end{aligned}
\]
Since $\delta_N=\lambda_N-1=O(N^{-3})$, replacing $\lambda_N^{-q-1}$ by $1$ for $q=0,1,2$ changes the expansion by only $O(N^{-6})$.  This proves \eqref{eq:delta-M0-M1}.
\end{proof}

\begin{proposition}\label{prop:schur-moments-order-five}
As $N\to\infty$,
\begin{align}
        M_0(N)&=\frac{2\zeta(3)}{N^3}
              +\frac{6\zeta(4)}{N^4}
              +\frac{16\zeta(5)}{N^5}
              +O(N^{-6}),\label{eq:M0-sharp-main}\\
        M_1(N)&=\frac{2\zeta(4)}{N^4}
              +\frac{8\zeta(5)}{N^5}
              +O(N^{-6}),\label{eq:M1-sharp-main}\\
        M_2(N)&=\frac{2\zeta(5)}{N^5}+O(N^{-6}).\label{eq:M2-sharp-main}
\end{align}
Consequently,
\begin{equation}\label{eq:lambda-order-five}
        \lambda_N
        =1+\frac{2\zeta(3)}{N^3}
          +\frac{8\zeta(4)}{N^4}
          +\frac{26\zeta(5)}{N^5}
          +O(N^{-6}).
\end{equation}
\end{proposition}

\begin{proof}
The moment identities are
\begin{equation}\label{eq:M0-formula-main}
        M_0(N)=\sum_{m\in\Z_*}\frac{v_m(\alpha)^2}{d_m},
\end{equation}
\begin{equation}\label{eq:M1-formula-main}
        M_1(N)=\sum_{m,\ell\in\Z_*}
        \frac{v_m(\alpha)v_\ell(\alpha)v_{m-\ell}(\alpha)}{d_m d_\ell},
\end{equation}
and
\begin{equation}\label{eq:M2-formula-main}
        M_2(N)=\sum_{m,\ell,p\in\Z_*}
        \frac{v_m(\alpha)v_p(\alpha)v_{m-\ell}(\alpha)v_{\ell-p}(\alpha)}{d_m d_\ell d_p}.
\end{equation}
These identities follow from \eqref{eq:b-K-entries} and the definition of $M_q$.  Appendix~\ref{app:moment-evaluations} evaluates these moments through order five.
Combining these evaluations with \Cref{lem:moment-expansion-main} yields
\[
\begin{aligned}
        \delta_N
        &=M_0(N)+M_1(N)+M_2(N)+O(N^{-6})\\
        &=\frac{2\zeta(3)}{N^3}
         +\frac{8\zeta(4)}{N^4}
         +\frac{26\zeta(5)}{N^5}
         +O(N^{-6}).
\end{aligned}
\]
This proves \eqref{eq:lambda-order-five}.
\end{proof}

\begin{proof}[Proof of \Cref{thm:order-five-asymptotics}]
By the critical-block localization in \Cref{thm:first-level-structure} and \Cref{prop:reciprocal-block},
\[
        \sigone(\OmN)=\frac1{\lambda_N}.
\]
Since $\lambda_N-1=O(N^{-3})$, \Cref{prop:schur-moments-order-five} gives
\[
        \frac1{\lambda_N}
        =1-(\lambda_N-1)+O((\lambda_N-1)^2)
        =1-\frac{2\zeta(3)}{N^3}
          -\frac{8\zeta(4)}{N^4}
          -\frac{26\zeta(5)}{N^5}
          +O(N^{-6}).
\]
\end{proof}

\appendix
% ============================================================

% ============================================================
\section{Nodal graphs and the Euler boundary bound}
\label{app:nodal-input}
% ============================================================

We separate the polygon-specific analytic input from the topological Euler count.  We cite two standard ingredients for the former: the local analytic structure of planar harmonic nodal sets \cite{Cheng1976} and the polyhomogeneous expansion for elliptic problems at polygonal corners \cite{Grisvard1985,CostabelDauge1993,CostabelDauge1994,Dauge1988,Kondratev1967}.  We then verify that the Robin asymptotics apply at the vertices of a regular polygon and that the global nodal set closes up to a finite properly embedded forest.  The Euler estimate is stated separately as an abstract lemma for a finite forest in a closed topological disk and is then applied to this polygonal nodal forest.

\begin{proposition}\label{prop:robin-corner-expansion-input}
Let
\[
        W_{\omega,r_0}=\{(r,\phi):0<r<r_0,\ 0<\phi<\omega\},
        \qquad 0<\omega<\pi,
\]
and let $u\in H^1(W_{\omega,r_0})$ be a nonzero real harmonic function satisfying $\partial_\nu u=\sigma u$ on the two sides.  Then either the leading corner term is a nonzero constant, or there exist $k\ge1$, $a_k\ne0$, and $\eps>0$ such that
\begin{equation}\label{eq:corner-leading-profile}
        u(r,\phi)
        =a_k r^{k\pi/\omega}\cos\!\left(\frac{k\pi\phi}{\omega}\right)
        +R(r,\phi),
\end{equation}
where, in side collars and compact closed subwedges,
\begin{equation}\label{eq:corner-C1-remainder}
        |\partial_r^a\partial_\phi^bR(r,\phi)|
        \le C_{a,b}r^{k\pi/\omega+\eps-a},
        \qquad a+b\le1.
\end{equation}
In the constant-leading case, the trace has a fixed nonzero sign on both adjacent sides sufficiently close to the vertex.  In the second case the two side traces have leading terms
\[
        a_k r^{k\pi/\omega},
        \qquad
        a_k(-1)^k r^{k\pi/\omega},
\]
and the interior nodal set consists near the vertex of finitely many analytic arcs asymptotic to the simple zero rays of the cosine factor.  Thus boundary zeros cannot accumulate at a vertex; moreover, every vanishing vertex is incident to an interior nodal arc.
\end{proposition}

\begin{proof}
In polar coordinates the Robin conditions are
\[
        -\partial_\phi u=\sigma r u\quad(\phi=0),
        \qquad
        \partial_\phi u=\sigma r u\quad(\phi=\omega).
\]
Thus the principal boundary operators are Neumann, while the Robin terms contain an additional factor of $r$ and are lower order in the corner calculus.  The cited polygonal-corner theory applies with principal angular pencil
\[
        -\Phi''=\nu^2\Phi,
        \qquad
        \Phi'(0)=\Phi'(\omega)=0.
\]
Its positive exponents are $\nu=k\pi/\omega$, with simple eigenfunctions $\cos(k\pi\phi/\omega)$.  Finite energy excludes negative exponents and the logarithmic zero mode.  To rule out a logarithm at the first positive exponent $\nu_0$, suppose that the highest logarithmic power is $J>0$, with coefficient $\Phi$ in the Neumann kernel.  The coefficient at the next logarithmic order would then have to solve
\[
        (\partial_\phi^2+\nu_0^2)\Psi=-2\nu_0J\Phi,
        \qquad \Psi'(0)=\Psi'(\omega)=0,
\]
Testing against $\Phi$, the Fredholm compatibility condition would require
\[
        -2\nu_0J\|\Phi\|_{L^2(0,\omega)}^2=0,
\]
which is impossible.  Hence the leading positive term has no logarithmic factor.  The cited corner theory now gives \eqref{eq:corner-leading-profile}--\eqref{eq:corner-C1-remainder}.  Evaluating at $\phi=0,\omega$ gives the side asymptotics.  Since the zeros of the leading angular factor are simple, the $C^1$ remainder and the implicit-function theorem yield the stated nodal arcs.
\end{proof}

\begin{lemma}\label{lem:finite-boundary-sign-components}
Let $u$ be a nonzero real Steklov eigenfunction on the regular polygon $\OmN$ associated with a positive eigenvalue.  Then $u$ has a continuous representative on $\overline{\OmN}$.  On each open side its trace is real analytic and is not identically zero on any nontrivial subsegment.  If an open-side boundary zero has order $q\ge1$, the first nonzero homogeneous harmonic term is, after flattening the side, a nonzero multiple of $\operatorname{Re}z^q$; in particular at least one nodal arc enters the polygon from that zero.
\end{lemma}

\begin{proof}
Standard analytic boundary regularity for the harmonic equation with analytic Robin data gives analyticity on every open side.  Continuity at the vertices follows from \Cref{prop:robin-corner-expansion-input}.  If the trace vanished on an open side segment, the Robin condition would also give a vanishing normal derivative there.  Analytic Cauchy uniqueness, equivalently local reflection followed by unique continuation, would force $u\equiv0$.  Thus open-side zeros are isolated.

At an open-side zero, flatten the side to $\{y=0\}$.  If $P_q$ is the first nonzero homogeneous term in the Taylor expansion, comparison of the lowest orders in $u_y=\pm\sigma u$ gives $\partial_yP_q(x,0)=0$.  A homogeneous harmonic polynomial with this Neumann condition is a nonzero multiple of $\operatorname{Re}z^q$.  Its zero rays enter the upper half-plane, and the standard local nodal-set structure gives the corresponding incident analytic arcs.
\end{proof}

\begin{corollary}\label{cor:finite-boundary-sign-components}
For every nonzero real positive-eigenvalue Steklov eigenfunction on $\OmN$, the boundary zero set is finite and nonempty.  Consequently the positive and negative boundary sign sets have finitely many connected components.  Moreover every boundary zero is incident to the interior nodal set.
\end{corollary}

\begin{proof}
Lemma~\ref{lem:finite-boundary-sign-components} gives local finiteness on open sides, and \Cref{prop:robin-corner-expansion-input} excludes accumulation at the finitely many vertices.  Compactness of the boundary then gives finiteness.  Since a positive-eigenvalue Steklov eigenfunction is orthogonal to constants on the boundary,
\[
        \int_{\partial\OmN}u\,ds=0,
\]
the nonzero continuous trace changes sign and therefore has a boundary zero.  Incidence with the interior nodal set follows from \Cref{lem:finite-boundary-sign-components,prop:robin-corner-expansion-input}.
\end{proof}

\begin{proposition}\label{prop:proper-nodal-forest}
Let $u$ be a nonzero real positive-eigenvalue Steklov eigenfunction on $\OmN$.  Let $\Gamma(u)$ be the closure in $\overline{\OmN}$ of the interior nodal set, with vertices at interior critical zeros and at boundary zeros.  Then $\Gamma(u)$ is a finite properly embedded forest.  Its boundary vertices are exactly the boundary zeros of the trace, and every tree component has at least two boundary vertices.
\end{proposition}

\begin{proof}
By the local nodal-set theorem for planar harmonic functions \cite{Cheng1976}, the interior zero set is locally a finite real-analytic graph: one analytic arc passes through a regular zero, while exactly $2q$ arcs meet at a zero of order $q\ge2$.  Near the polygonal boundary, the corresponding local finiteness and incidence statements follow from \Cref{lem:finite-boundary-sign-components,prop:robin-corner-expansion-input,cor:finite-boundary-sign-components}.  Compactness of $\overline{\OmN}$ provides a finite subcover by these local graph neighborhoods.  Hence the closed nodal set is a finite embedded graph, and its boundary vertices are precisely the boundary zeros.

The graph is acyclic.  Indeed, any cycle in a finite embedded graph contains a simple cycle whose image is a Jordan curve $J$ in the zero set.  Since $\OmN$ is convex, the bounded Jordan domain enclosed by $J$ lies in $\OmN$ (possible contact points of $J$ with $\partial\OmN$ are boundary points of that Jordan domain).  Since $u=0$ on $J$, the maximum principle gives $u=0$ in the enclosed domain; unique continuation then yields $u\equiv0$ in $\OmN$, a contradiction.  Thus every connected component of $\Gamma(u)$ is a tree.

Every nontrivial finite tree has at least two leaves.  No leaf lies in the interior: a regular interior zero has degree $2$, and a zero of order $q\ge2$ has degree $2q\ge4$.  Thus all leaves lie on $\partial\OmN$, and each tree component has at least two boundary vertices.
\end{proof}

\begin{lemma}\label{lem:topological-euler-boundary-count}
Let $D$ be a closed topological disk and let $\Gamma\subset D$ be a nonempty finite forest properly embedded in $D$, with no edge contained in $\partial D$.  Let $B$ be the number of boundary vertices of $\Gamma$, let $C$ be the number of tree components, and let $F$ be the number of connected components of $\operatorname{int}(D)\setminus\Gamma$.  Assume that every tree component has at least two boundary vertices.  Then
\begin{equation}\label{eq:topological-euler-boundary-estimate}
        B\le2(F-1).
\end{equation}
\end{lemma}

\begin{proof}
Let $V_i$ be the number of interior vertices and $E_\Gamma$ the number of edges of $\Gamma$.  Since a tree has one fewer edge than vertices, summing over the $C$ tree components gives
\[
        E_\Gamma=(V_i+B)-C.
\]
Adjoin the boundary circle $\partial D$ and subdivide it at the $B$ boundary vertices.  This adds exactly $B$ boundary edges and no new vertices.  The augmented graph
\[
        G:=\Gamma\cup\partial D
\]
is connected because every tree component meets $\partial D$.  Hence
\[
        V=V_i+B,
        \qquad
        E=E_\Gamma+B.
\]
Its faces in the disk are exactly the connected components of $\operatorname{int}(D)\setminus\Gamma$, so their number is $F$.  Euler's formula for a connected finite graph embedded in a closed topological disk gives
\[
        V-E+F=1.
\]
Equivalently, after capping the boundary by one exterior face, this is the spherical identity
\[
        V-E+(F+1)=2.
\]
Substituting the vertex and edge counts yields
\[
\begin{aligned}
        F
        &=1-V+E\\
        &=1-(V_i+B)+(E_\Gamma+B)\\
        &=1+B-C.
\end{aligned}
\]
By hypothesis, each of the $C$ tree components has at least two boundary vertices, and distinct components have disjoint vertex sets.  Thus $B\ge2C$.  Consequently
\[
        F-1=B-C\ge B/2,
\]
which is \eqref{eq:topological-euler-boundary-estimate}.
\end{proof}

\begin{corollary}\label{cor:polygonal-euler-boundary-count}
Let $K(u)$ be the number of interior nodal domains and let $N_\partial(u)$ be the number of boundary nodal domains of a nonzero real positive-eigenvalue Steklov eigenfunction on $\OmN$.  Then
\begin{equation}\label{eq:polygonal-euler-boundary-estimate}
        N_\partial(u)\le2\bigl(K(u)-1\bigr).
\end{equation}
\end{corollary}

\begin{proof}
By \Cref{cor:finite-boundary-sign-components}, the boundary zero set is finite and nonempty, and its points are precisely the boundary vertices of the nodal forest $\Gamma(u)$.  Since the nonzero trace changes sign on the boundary circle, those points cut $\partial\OmN$ into exactly $N_\partial(u)$ sign arcs.  Hence the number $B$ of boundary vertices of $\Gamma(u)$ is
\[
        B=N_\partial(u).
\]
By \Cref{prop:proper-nodal-forest}, $\Gamma(u)$ is a finite properly embedded forest and every tree component has at least two boundary vertices.  Its complementary components in $\OmN$ are exactly the interior nodal domains, so $F=K(u)$.  Applying \Cref{lem:topological-euler-boundary-count} gives the result.
\end{proof}

% ============================================================
\section{Schur resolvent estimates}
\label{app:schur-resolvent}
% ============================================================

We record the coefficient and resolvent estimates used in \Cref{sec:schur-asymptotics}.  Throughout this appendix, $\alpha=1/N$ and $0<\alpha\le\alpha_0$, where $0<\alpha_0<1/10$ is fixed.

\begin{lemma}\label{lem:uniform-coeff-bounds}
For $0<\alpha\le\alpha_0$ and $j\ge1$,
\begin{equation}\label{eq:vj-uniform-bound}
        0<v_j(\alpha)\le C_{\alpha_0}\,\alpha\,j^{-1+2\alpha_0}.
\end{equation}
Set $A_j:=2H_j-j^{-1}$, where $H_j=\sum_{q=1}^j q^{-1}$.  Uniformly for $j\ge1$ and $0<\alpha\le\alpha_0$,
\begin{align}
        v_j(\alpha)
        &=\frac{\alpha}{j}
        +\frac{\alpha^2A_j}{j}
        +O_{\alpha_0}\!\left(
          \frac{\alpha^3(1+\log j)^2}{j^{1-2\alpha_0}}
        \right),\label{eq:vj-first-expansion}\\
        v_j(\alpha)
        &=\frac{\alpha}{j}
        \left[
          1+\alpha A_j+
          \alpha^2\left(\frac{A_j^2}{2}+\frac1{2j^2}\right)
        \right]
        +O_{\alpha_0}\!\left(
          \frac{\alpha^4(1+\log j)^3}{j^{1-2\alpha_0}}
        \right).\label{eq:vj-second-expansion}
\end{align}
\end{lemma}

\begin{proof}
From the Gamma formula in \Cref{lem:normalized-fourier-coefficients}, for $j\ge1$,
\begin{equation}\label{eq:v-product}
        v_j(\alpha)
        =\frac{\alpha}{j-\alpha}
        \prod_{q=1}^{j-1}\frac{q+\alpha}{q-\alpha}.
\end{equation}
For $q\ge1$ and $0<\alpha\le\alpha_0<1/2$,
\[
        \log\frac{q+\alpha}{q-\alpha}
        =\frac{2\alpha}{q}+O_{\alpha_0}\!\left(\frac{\alpha^3}{q^3}\right).
\]
Summation gives
\[
        \prod_{q=1}^{j-1}\frac{q+\alpha}{q-\alpha}
        \le C_{\alpha_0}j^{2\alpha}
        \le C_{\alpha_0}j^{2\alpha_0},
\]
and \eqref{eq:vj-uniform-bound} follows from $\alpha/(j-\alpha)\le C_{\alpha_0}\alpha/j$.

Set
\[
        x_j(\alpha):=\log\frac{v_j(\alpha)}{\alpha/j}.
\]
Expanding the endpoint factor and the logarithms in \eqref{eq:v-product} gives
\begin{equation}\label{eq:log-vj-second-clean}
        x_j(\alpha)
        =\alpha A_j+\frac{\alpha^2}{2j^2}
        +O_{\alpha_0}\!\left(
          \alpha^3\left(\frac1{j^3}+\sum_{q=1}^{j-1}q^{-3}\right)
        \right).
\end{equation}
In particular, $x_j(\alpha)=\alpha A_j+O_{\alpha_0}(\alpha^2)$, while the product estimate above gives $e^{|x_j(\alpha)|}\le C_{\alpha_0}j^{2\alpha_0}$.  Applying
\[
        e^x=1+x+O(x^2e^{|x|})
\]
to the rough form of \eqref{eq:log-vj-second-clean} yields \eqref{eq:vj-first-expansion}.  Applying instead
\[
        e^x=1+x+\frac{x^2}{2}+O(|x|^3e^{|x|})
\]
to the refined form and multiplying by $\alpha/j$ yields \eqref{eq:vj-second-expansion}.
\end{proof}

\begin{lemma}\label{lem:tail-bounds}
As $N\to\infty$,
\begin{equation}\label{eq:tail-bounds}
        \|K_N\|=O(N^{-1}),
        \qquad
        \|b_N\|=O(N^{-3/2}).
\end{equation}
\end{lemma}

\begin{proof}
We first estimate $b_N$.
For $m=\pm j$, $j\ge1$,
\[
        d_j=jN+1=\frac{j+\alpha}{\alpha},
        \qquad
        d_{-j}=jN-1=\frac{j-\alpha}{\alpha}.
\]
Using \eqref{eq:vj-uniform-bound},
\[
\begin{aligned}
        \|b_N\|^2
        &=\sum_{m\in\Z_*}\frac{v_m(\alpha)^2}{d_m}  \\
        &\le C\sum_{j\ge1}\alpha^2j^{-2+4\alpha_0}
        \left(\frac{\alpha}{j+\alpha}+\frac{\alpha}{j-\alpha}\right)
        \le C\alpha^3\sum_{j\ge1}j^{-3+4\alpha_0}=O(\alpha^3).
\end{aligned}
\]
Thus $\|b_N\|=O(\alpha^{3/2})$.

For $K_N$, it suffices to estimate the Hilbert--Schmidt norm.
The diagonal contribution is
\[
        \sum_{m\in\Z_*}\frac{v_0^2}{d_m^2}
        =\sum_{m\in\Z_*}\frac1{d_m^2}=O(\alpha^2).
\]
For $m\ne n$, \eqref{eq:vj-uniform-bound} gives
\[
\begin{aligned}
        \sum_{m\ne n}\frac{v_{m-n}(\alpha)^2}{d_m d_n}
        &\le C\alpha^4
        \sum_{m\ne n}\frac{|m-n|^{-2+4\alpha_0}}{|m|\,|n|}.
\end{aligned}
\]
The double sum is finite for $\alpha_0<1/4$.  Writing $j=m-n$ and applying \eqref{eq:convolution-bound}, we bound it by
\[
        C\sum_{j\ne0}|j|^{-2+4\alpha_0}\frac{\log(e+|j|)}{1+|j|}<\infty.
\]
Hence $\|K_N\|_{\HS}^2=O(\alpha^2)$ and $\|K_N\|=O(\alpha)$.
\end{proof}

% ============================================================
\section{Moment and Euler-sum evaluations}
\label{app:moment-evaluations}
% ============================================================

This appendix evaluates the Schur moments through order five.  The computations are independent of the monotonicity proof.

\subsection{Lower-order terms}

\begin{lemma}\label{lem:euler-sums-basic}
The following classical Euler sums hold:
\begin{equation}\label{eq:euler-basic}
        \sum_{j=1}^{\infty}\frac1{j^s}=\zeta(s),\qquad s>1,
        \qquad
        \sum_{j=1}^{\infty}\frac{H_j}{j^3}=\frac54\zeta(4).
\end{equation}
Moreover, for every sufficiently small $\eta>0$, the sums obtained by inserting a factor $(1+\log j)^qj^{-\eta}$ into the remainders below are convergent.
\end{lemma}

\begin{proof}
The first identity is the definition of $\zeta(s)$.  The second follows from Euler's evaluation
\[
        \sum_{j\ge1}\frac{H_j}{j^p}
        =\frac{p+2}{2}\zeta(p+1)
        -\frac12\sum_{a=1}^{p-2}\zeta(p-a)\zeta(a+1),
        \qquad p\ge2,
\]
with $p=3$.
The final assertion follows by comparison with $\sum j^{-1-\eta/2}$, after absorbing powers of $\log j$ into the constant.
\end{proof}

\begin{lemma}\label{lem:M0-expansion}
As $N\to\infty$,
\begin{equation}\label{eq:M0-expansion-two}
        M_0(N)=\frac{2\zeta(3)}{N^3}+\frac{6\zeta(4)}{N^4}+O(N^{-5}).
\end{equation}
\end{lemma}

\begin{proof}
Pair the terms $m=j$ and $m=-j$ in \eqref{eq:M0-formula-main} to obtain
\begin{equation}\label{eq:M0-paired}
        M_0(N)=\sum_{j\ge1}v_j(\alpha)^2
        \left(\frac1{jN+1}+\frac1{jN-1}\right).
\end{equation}
The denominator factor satisfies
\begin{equation}\label{eq:D1-expansion}
        \frac1{jN+1}+\frac1{jN-1}
        =\frac{2\alpha j}{j^2-\alpha^2}
        =\frac{2\alpha}{j}+O\left(\frac{\alpha^3}{j^3}\right).
\end{equation}
By \eqref{eq:vj-first-expansion},
\[
        v_j(\alpha)^2
        =\frac{\alpha^2}{j^2}
        +\frac{2\alpha^3}{j^2}\left(2H_j-\frac1j\right)
        +O\left(\frac{\alpha^4(1+\log j)^4}{j^{2-4\alpha_0}}\right).
\]
Multiplying by \eqref{eq:D1-expansion} and summing yields
\[
\begin{aligned}
        M_0(N)
        &=2\alpha^3\sum_{j\ge1}\frac1{j^3}
          +4\alpha^4\sum_{j\ge1}
          \left(\frac{2H_j}{j^3}-\frac1{j^4}\right)
          +O(\alpha^5) \\
        &=2\zeta(3)\alpha^3
          +4\left(2\cdot\frac54\zeta(4)-\zeta(4)\right)\alpha^4
          +O(\alpha^5),
\end{aligned}
\]
which proves \eqref{eq:M0-expansion-two}.
\end{proof}

\begin{lemma}\label{lem:M1-expansion}
As $N\to\infty$,
\begin{equation}\label{eq:M1-expansion-two}
        M_1(N)=\frac{2\zeta(4)}{N^4}+O(N^{-5}).
\end{equation}
\end{lemma}

\begin{proof}
Split \eqref{eq:M1-formula-main} into the diagonal part $m=\ell$ and the off-diagonal part $m\ne\ell$.  On the diagonal, $v_{m-\ell}=v_0=1$, so
\[
        M_1^{\mathrm{diag}}(N)
        =\sum_{m\in\Z_*}\frac{v_m(\alpha)^2}{d_m^2}.
\]
Pairing $m=\pm j$ gives
\[
        M_1^{\mathrm{diag}}(N)
        =\sum_{j\ge1}v_j(\alpha)^2
        \left(\frac1{(jN+1)^2}+\frac1{(jN-1)^2}\right).
\]
Since
\[
        \frac1{(jN+1)^2}+\frac1{(jN-1)^2}
        =\frac{2\alpha^2}{j^2}+O\left(\frac{\alpha^4}{j^4}\right)
\]
and $v_j(\alpha)^2=\alpha^2j^{-2}+O(\alpha^3(1+\log j)j^{-2})$, we obtain
\[
        M_1^{\mathrm{diag}}(N)=2\alpha^4\sum_{j\ge1}\frac1{j^4}+O(\alpha^5)
        =2\zeta(4)\alpha^4+O(\alpha^5).
\]
For the off-diagonal part, \eqref{eq:vj-uniform-bound} and $d_m^{-1}\le C\alpha |m|^{-1}$ imply
\[
\begin{aligned}
        |M_1^{\mathrm{off}}(N)|
        &\le C\alpha^5
        \sum_{m,\ell\in\Z_*\,:\,m\ne\ell}
        \frac1{|m|^{2-2\alpha_0}|\ell|^{2-2\alpha_0}|m-\ell|^{1-2\alpha_0}}
        =O(\alpha^5),
\end{aligned}
\]
The double sum is finite for sufficiently small $\alpha_0$.  Set $a=2-2\alpha_0>1$ and $b=1-2\alpha_0$.  For fixed $j=m-\ell\ne0$,
\[
        \sum_{\ell\in\Z_*:\,\ell+j\ne0}
        \frac1{|\ell+j|^a|\ell|^a}
        \le \frac{C}{|j|^a},
\]
by splitting into neighborhoods of $0$ and $-j$ and the complementary tail.  Therefore the double sum is bounded by
\[
        C\sum_{j\ne0}\frac1{|j|^{a+b}}<\infty,
        \qquad a+b=3-4\alpha_0>1.
\]
This proves \eqref{eq:M1-expansion-two}.
\end{proof}

\subsection{Weight-five terms}

We now compute the order-five coefficient in \Cref{thm:order-five-asymptotics}.  Throughout this subsection, \(\alpha=N^{-1}\), \(d_m=|1+mN|\), and the constants in \(O(\cdot)\) are independent of \(N\).  Fix \(0<\alpha_0<1/10\); all estimates are uniform for \(0<\alpha\le\alpha_0\).

\begin{lemma}\label{lem:euler-sums-order-five}
Let \(H_j=\sum_{q=1}^j q^{-1}\) and \(H_j^{(2)}=\sum_{q=1}^j q^{-2}\).  Then
\begin{align}
        \sum_{j\ge1}\frac{H_j}{j^4}
        &=3\zeta(5)-\zeta(2)\zeta(3),\label{eq:euler-H-j4}\\
        \sum_{j\ge1}\frac{H_j^2}{j^3}
        &=\frac72\zeta(5)-\zeta(2)\zeta(3),\label{eq:euler-H2-j3}\\
        \sum_{j\ge1}\frac{H_j^{(2)}}{j^3}
        &=3\zeta(2)\zeta(3)-\frac92\zeta(5).\label{eq:euler-H2second-j3}
\end{align}
Moreover,
\begin{equation}\label{eq:lattice-sum-order-five}
        \sum_{\substack{m,\ell\in\Z_*\\m\ne\ell}}
        \frac1{|m|^2|\ell|^2|m-\ell|}
        =8\zeta(2)\zeta(3)-12\zeta(5).
\end{equation}
\end{lemma}

\begin{proof}
We use the standard multiple-zeta notation
\[
        \zeta(s_1,\ldots,s_k)
        =\sum_{n_1>\cdots>n_k\ge1}
        \frac1{n_1^{s_1}\cdots n_k^{s_k}},
        \qquad s_1\ge2.
\]
We recall the weight-five reduction.  Euler's formula gives
\[
        \zeta(4,1)=2\zeta(5)-\zeta(2)\zeta(3),
\]
and duality gives \(\zeta(3,1,1)=\zeta(4,1)\).  The depth-two sum and shuffle relations at weight five are
\begin{equation}\label{eq:mzv-weight-five-sum-relation}
        \zeta(4,1)+\zeta(3,2)+\zeta(2,3)=\zeta(5),
\end{equation}
and
\begin{equation}\label{eq:mzv-weight-five-shuffle-relation}
        \zeta(2)\zeta(3)
        =6\zeta(4,1)+3\zeta(3,2)+\zeta(2,3).
\end{equation}
Solving \eqref{eq:mzv-weight-five-sum-relation}--\eqref{eq:mzv-weight-five-shuffle-relation} with the value of \(\zeta(4,1)\) gives
\begin{equation}\label{eq:mzv-weight-five-values}
        \zeta(4,1)=\zeta(3,1,1)=2\zeta(5)-\zeta(2)\zeta(3),
        \qquad
        \zeta(3,2)=3\zeta(2)\zeta(3)-\frac{11}{2}\zeta(5).
\end{equation}
These are classical Euler reductions; the derivation is included to make the normalizations explicit.  See, for example, \cite[Sections~2--3]{BorweinBradleyBroadhurst1997}.  Separating the terms for which the inner index equals the outer index gives
\[
        \sum_{j\ge1}\frac{H_j}{j^4}=\zeta(4,1)+\zeta(5),
        \qquad
        \sum_{j\ge1}\frac{H_j^{(2)}}{j^3}=\zeta(3,2)+\zeta(5).
\]
Similarly,
\[
        \sum_{j\ge1}\frac{H_j^2}{j^3}
        =2\zeta(3,1,1)+\zeta(3,2)+2\zeta(4,1)+\zeta(5).
\]
Substitution of \eqref{eq:mzv-weight-five-values} proves \eqref{eq:euler-H-j4}--\eqref{eq:euler-H2second-j3}.

It remains to evaluate the lattice sum.  Split the double sum according to the signs of the two indices.  The same-sign contribution is
\[
        4A,
        \qquad
        A:=\sum_{l,k\ge1}\frac1{l^2(l+k)^2k},
\]
and the opposite-sign contribution is
\[
        2T,
        \qquad
        T:=\sum_{m,n\ge1}\frac1{m^2n^2(m+n)}.
\]
For \(T\), the partial fraction identity
\[
        \frac1{m^2n^2(m+n)}
        =\frac1{m^3n^2}-\frac1{m^4n}+\frac1{m^4(m+n)}
\]
gives, after summing over \(n\) and using \(\sum_{n\ge1}(n^{-1}-(m+n)^{-1})=H_m\),
\[
        T=\zeta(2)\zeta(3)-\sum_{m\ge1}\frac{H_m}{m^4}
        =2\zeta(2)\zeta(3)-3\zeta(5).
\]
For \(A\), use
\[
        \frac1{l^2(l+k)^2}
        =\frac1{k^2}\left(\frac1{l^2}+\frac1{(l+k)^2}\right)
         -\frac2{k^3}\left(\frac1l-\frac1{l+k}\right).
\]
Summing over \(l\) gives
\[
        A=\sum_{k\ge1}\left[
        \frac{2\zeta(2)-H_k^{(2)}}{k^3}
        -\frac{2H_k}{k^4}\right]
        =\zeta(2)\zeta(3)-\frac32\zeta(5),
\]
where we used \eqref{eq:euler-H-j4} and \eqref{eq:euler-H2second-j3}.  Therefore \(4A+2T=8\zeta(2)\zeta(3)-12\zeta(5)\), which proves \eqref{eq:lattice-sum-order-five}.
\end{proof}

\begin{lemma}\label{lem:discrete-convolution-estimates}
Let $a>1$ and $0<b<1$.  There is a constant $C=C(a,b)$ such that, for every $j\in\Z_*$,
\begin{equation}\label{eq:conv-ab}
        \sum_{n\in\Z_*:\,n\ne j}
        \frac1{|n|^a|n-j|^b}
        \le C|j|^{-b}.
\end{equation}
If $a>1$, then there is $C=C(a)$ such that
\begin{equation}\label{eq:conv-log-aa}
        \sum_{n\in\Z_*:\,n\ne j}
        \frac{1+\log(e+|n|)+\log(e+|n-j|)}{|n|^a|n-j|^a}
        \le C\frac{1+\log(e+|j|)}{|j|^a}.
\end{equation}
\end{lemma}

\begin{proof}
We prove \eqref{eq:conv-ab}.  The logarithmic estimate follows from the same decomposition, with the logarithmic factors retained.  Split the summation into
\[
        |n|\le |j|/2,
        \qquad |n-j|\le |j|/2,
        \qquad |n|>|j|/2\text{ and }|n-j|>|j|/2.
\]
In the first region, $|n-j|\ge |j|/2$, and the contribution is bounded by $C|j|^{-b}\sum_{n\ne0}|n|^{-a}\le C|j|^{-b}$.  In the second region, $|n|\ge |j|/2$, so the contribution is bounded by
\[
        C|j|^{-a}\sum_{1\le |m|\le 3|j|/2}|m|^{-b}
        \le C|j|^{-a}|j|^{1-b}
        \le C|j|^{-b},
\]
because $a>1$.  In the complementary region, both denominators are bounded below by fixed multiples of $|n|+|j|$.  Dyadic summation gives $C|j|^{1-a-b}\le C|j|^{-b}$, proving \eqref{eq:conv-ab}.  For \eqref{eq:conv-log-aa}, the first two regions give $C(1+\log(e+|j|))|j|^{-a}$ times a convergent $a$-series, and the dyadic tail satisfies the same bound.
\end{proof}

\begin{lemma}\label{lem:M1-off-order-five}
Let
\[
        M_1^{\mathrm{off}}(N)
        :=\sum_{\substack{m,\ell\in\Z_*\\m\ne\ell}}
        \frac{v_m(\alpha)v_\ell(\alpha)v_{m-\ell}(\alpha)}{d_md_\ell}.
\]
Then
\begin{equation}\label{eq:M1-off-order-five}
        M_1^{\mathrm{off}}(N)
        =\alpha^5
        \sum_{\substack{m,\ell\in\Z_*\\m\ne\ell}}
        \frac1{|m|^2|\ell|^2|m-\ell|}
        +O(\alpha^6).
\end{equation}
\end{lemma}

\begin{proof}
For \(q\in\Z_*\), \Cref{lem:uniform-coeff-bounds} and \eqref{eq:vj-first-expansion} give
\begin{equation}\label{eq:vq-first-order-for-M1off}
        v_q(\alpha)=\frac{\alpha}{|q|}
        +O_{\alpha_0}\left(
        \frac{\alpha^2(1+\log(e+|q|))}{|q|^{1-2\alpha_0}}
        \right),
\end{equation}
and the denominator satisfies
\begin{equation}\label{eq:dq-inverse-first-order}
        d_q^{-1}=\frac{\alpha}{|q|}+O\left(\frac{\alpha^2}{|q|^2}\right).
\end{equation}
Write the five factors in each summand as \(F_i=L_i+R_i\), where \(L_i\) is the leading term in \eqref{eq:vq-first-order-for-M1off} or \eqref{eq:dq-inverse-first-order}.  The product of the five leading terms gives the summand in \eqref{eq:M1-off-order-five}.  For the remainder, use the telescoping identity
\[
        \prod_{i=1}^{5}F_i-\prod_{i=1}^{5}L_i
        =\sum_{s=1}^{5}R_s
        \left(\prod_{i<s}F_i\right)
        \left(\prod_{i>s}L_i\right).
\]
The bounds in \Cref{lem:uniform-coeff-bounds}, \eqref{eq:vq-first-order-for-M1off}, and \eqref{eq:dq-inverse-first-order} control each full factor \(F_i\) by the same power as its leading term, up to the loss \(|q|^{2\alpha_0}\).  Only the remainder factor \(R_s\) contributes an additional logarithm.  Hence the total remainder is bounded by
\[
        C\alpha^6
        \frac{1+\log(e+|m|)+\log(e+|\ell|)+\log(e+|m-\ell|)}
        {|m|^{2-2\alpha_0}|\ell|^{2-2\alpha_0}|m-\ell|^{1-2\alpha_0}}.
\]
This majorant is summable for \(\alpha_0<1/10\).  Set \(j=m-\ell\) and \(a=2-2\alpha_0>1\).  Then \eqref{eq:conv-log-aa} gives
\[
        \sum_{\ell\in\Z_*:\,\ell+j\ne0}
        \frac{1+\log(e+|\ell|)+\log(e+|\ell+j|)}
        {|\ell|^a|\ell+j|^a}
        \le C\frac{1+\log(e+|j|)}{|j|^a}.
\]
The majorant is therefore bounded by
\[
        C\sum_{j\ne0}\frac{(1+\log(e+|j|))^2}{|j|^{3-4\alpha_0}}<\infty.
\]
Therefore the termwise expansion is justified, with an \(O(\alpha^6)\) remainder.
\end{proof}

\begin{lemma}\label{lem:order-six-convolution-bound}
Define
\[
        R_2^{\mathrm{off}}(\alpha)
        :=\sum_{\substack{m,\ell,p\in\Z_*\\ \text{not }(m=\ell=p)}}
        \frac{v_m(\alpha)v_p(\alpha)v_{m-\ell}(\alpha)v_{\ell-p}(\alpha)}{d_md_\ell d_p},
\]
where \(v_0(\alpha)=1\).  Then
\[
        R_2^{\mathrm{off}}(\alpha)=O(\alpha^6).
\]
\end{lemma}

\begin{proof}
If \(m=\ell\ne p\), then \(v_{m-\ell}=v_0=1\).  By \Cref{lem:uniform-coeff-bounds} and \(d_q^{-1}\le C\alpha |q|^{-1}\) for \(q\in\Z_*\), this contribution is bounded by
\[
        C\alpha^6
        \sum_{m\ne p}
        \frac1{|m|^{3-2\alpha_0}|p|^{2-2\alpha_0}|m-p|^{1-2\alpha_0}}<\infty.
\]
Apply \eqref{eq:conv-ab} with \(a=3-2\alpha_0\) and \(b=1-2\alpha_0\), and then sum in the remaining variable.  The double sum is finite.  The case \(\ell=p\ne m\) is identical after relabeling.

It remains to treat \(m\ne\ell\) and \(\ell\ne p\), for which all four coefficient links are nonzero.  The absolute value of this part is at most
\[
        C\alpha^7
        \sum_{\substack{m,\ell,p\in\Z_*\\ m\ne\ell,\ \ell\ne p}}
        \frac1{|m|^{2-2\alpha_0}|p|^{2-2\alpha_0}|\ell|}
        \frac1{|m-\ell|^{1-2\alpha_0}|\ell-p|^{1-2\alpha_0}}.
\]
For fixed \(\ell\), apply \eqref{eq:conv-ab} with \(a=2-2\alpha_0\) and \(b=1-2\alpha_0\).  For completeness, split the sum in \(m\) into
\[
        |m|\le |\ell|/2,
        \qquad |m-\ell|\le |\ell|/2,
        \qquad |m|>|\ell|/2\text{ and }|m-\ell|>|\ell|/2.
\]
The first two regions contribute \(C|\ell|^{-1+2\alpha_0}\) and \(C|\ell|^{-1}\), respectively.  The remaining region contributes \(C|\ell|^{-2+4\alpha_0}\).  Each term is bounded by \(C|\ell|^{-1+2\alpha_0}\), and the same estimate holds for the sum in \(p\).  Therefore the triple sum is bounded by
\[
        C\sum_{\ell\ne0}|\ell|^{-1}|\ell|^{-2+4\alpha_0}<\infty
\]
for \(\alpha_0<1/4\).  This contribution is \(O(\alpha^7)\), and hence \(O(\alpha^6)\).  The cases above exhaust the complement of the fully diagonal chain \(m=\ell=p\).
\end{proof}

\begin{lemma}\label{lem:moment-sharper}
As \(N\to\infty\),
\begin{align}
        M_0(N)&=\frac{2\zeta(3)}{N^3}
              +\frac{6\zeta(4)}{N^4}
              +\frac{16\zeta(5)}{N^5}
              +O(N^{-6}),\label{eq:M0-sharp}\\
        M_1(N)&=\frac{2\zeta(4)}{N^4}
              +\frac{8\zeta(5)}{N^5}
              +O(N^{-6}),\label{eq:M1-sharp}\\
        M_2(N)&=\frac{2\zeta(5)}{N^5}+O(N^{-6}).\label{eq:M2-sharp}
\end{align}
\end{lemma}

\begin{proof}
For \(M_0\), combine \eqref{eq:vj-second-expansion} with
\[
        \frac1{jN+1}+\frac1{jN-1}
        =\frac{2\alpha}{j}+\frac{2\alpha^3}{j^3}+O\left(\frac{\alpha^5}{j^5}\right).
\]
As in \Cref{lem:M0-expansion}, the coefficients of \(\alpha^3\) and \(\alpha^4\) are \(2\zeta(3)\) and \(6\zeta(4)\).  The coefficient of \(\alpha^5\) is
\[
        \sum_{j\ge1}\left(\frac{4A_j^2}{j^3}+\frac{4}{j^5}\right).
\]
Since \(A_j=2H_j-j^{-1}\),
\[
        \frac{4A_j^2}{j^3}+\frac4{j^5}
        =\frac{16H_j^2}{j^3}-\frac{16H_j}{j^4}+\frac8{j^5}.
\]
Equations \eqref{eq:euler-H2-j3} and \eqref{eq:euler-H-j4} give
\[
        16\left(\frac72\zeta(5)-\zeta(2)\zeta(3)\right)
        -16\left(3\zeta(5)-\zeta(2)\zeta(3)\right)+8\zeta(5)
        =16\zeta(5).
\]
We now control the omitted terms.  After multiplication by the leading denominator term, the remainder in \eqref{eq:vj-second-expansion} contributes at most
\[
        C\alpha^6\sum_{j\ge1}\frac{(1+\log j)^3}{j^{3-2\alpha_0}}<\infty.
\]
The next term in the denominator expansion, together with the retained leading coefficient expansion, contributes a convergent bound of the form
\[
        C\alpha^7\sum_{j\ge1}\frac{(1+\log j)^2}{j^{5-2\alpha_0}}<\infty,
\]
and the displayed denominator remainder gives
\[
        C\alpha^7\sum_{j\ge1}\frac1{j^7}<\infty.
\]
Thus all omitted terms are \(O(\alpha^6)\), proving \eqref{eq:M0-sharp}.

For \(M_1\), decompose \(M_1=M_1^{\mathrm{diag}}+M_1^{\mathrm{off}}\) according to whether \(m=\ell\) or \(m\ne\ell\) in \eqref{eq:M1-formula-main}.  The diagonal part is
\[
        M_1^{\mathrm{diag}}(N)
        =\sum_{j\ge1}v_j(\alpha)^2
        \left(\frac1{(jN+1)^2}+\frac1{(jN-1)^2}\right).
\]
Using
\[
        \frac1{(jN+1)^2}+\frac1{(jN-1)^2}
        =\frac{2\alpha^2}{j^2}+O\left(\frac{\alpha^4}{j^4}\right)
\]
and
\[
        v_j(\alpha)^2
        =\frac{\alpha^2}{j^2}\left(1+2\alpha A_j\right)
        +O_{\alpha_0}\left(
        \frac{\alpha^4(1+\log j)^2}{j^{2-2\alpha_0}}
        \right),
\]
we obtain
\[
        M_1^{\mathrm{diag}}(N)
        =2\alpha^4\sum_{j\ge1}\frac1{j^4}
        +4\alpha^5\sum_{j\ge1}\frac{A_j}{j^4}
        +O(\alpha^6).
\]
Since
\[
        4\sum_{j\ge1}\frac{A_j}{j^4}
        =8\sum_{j\ge1}\frac{H_j}{j^4}-4\zeta(5)
        =20\zeta(5)-8\zeta(2)\zeta(3),
\]
we have
\[
        M_1^{\mathrm{diag}}(N)
        =2\zeta(4)\alpha^4
        +\bigl(20\zeta(5)-8\zeta(2)\zeta(3)\bigr)\alpha^5
        +O(\alpha^6).
\]
Lemma~\ref{lem:M1-off-order-five} and \eqref{eq:lattice-sum-order-five} give
\[
        M_1^{\mathrm{off}}(N)
        =\bigl(8\zeta(2)\zeta(3)-12\zeta(5)\bigr)\alpha^5+O(\alpha^6).
\]
Adding the two contributions proves \eqref{eq:M1-sharp}.

For \(M_2\), the fully diagonal chain \(m=\ell=p\) gives
\[
        \sum_{m\in\Z_*}\frac{v_m(\alpha)^2}{d_m^3}
        =2\alpha^5\sum_{j\ge1}\frac1{j^5}+O(\alpha^6)
        =2\zeta(5)\alpha^5+O(\alpha^6).
\]
The remaining chains form \(R_2^{\mathrm{off}}(\alpha)\), which is \(O(\alpha^6)\) by \Cref{lem:order-six-convolution-bound}.  This proves \eqref{eq:M2-sharp}.
\end{proof}

% ============================================================
\section*{Statement on the use of AI tools}
Large language models were used during the preparation of this manuscript for purposes including idea exploration, language editing, and related assistance. The authors reviewed the final manuscript and take full responsibility for its content, mathematical correctness, and conclusions.

% ============================================================
\raggedbottom

% ============================================================
\end{document}